\documentclass{article}
\usepackage{graphicx} 
\usepackage[a4paper, total={6.5in, 9.5in}]{geometry}
\usepackage{amsmath,amsthm,amssymb}
\usepackage{mathtools}
\usepackage[dvipsnames,table]{xcolor} 
\usepackage{setspace}
\usepackage{rotating} 
\usepackage{pdflscape} 
\usepackage[para,online,flushleft]{threeparttable}
 \usepackage{booktabs}
 \usepackage{arydshln}
 \usepackage{natbib}
 \usepackage{subcaption}  
\usepackage{algpseudocode} 
\usepackage{authblk}
\usepackage{url}
\usepackage{pifont}
\usepackage{amsfonts}
\usepackage{comment}
\usepackage{algorithm2e}
\usepackage{todonotes}

\usepackage{hyperref}
\hypersetup{
    colorlinks=true,
    citecolor=blue,
    linkcolor=blue
    }

\newtheorem{definition}{Definition}

\newtheorem{theorem}{Theorem}[section]
\newtheorem{corollary}[theorem]{Corollary}
\newtheorem{lemma}[theorem]{Lemma}

\usepackage{cleveref}

\crefname{observation}{observation}{observations}
\Crefname{observation}{Observation}{Observations}

\usepackage{tikz}
\usetikzlibrary{backgrounds}
\usetikzlibrary{calc}
\usetikzlibrary {arrows.meta}
\usetikzlibrary{decorations.pathmorphing}
\usetikzlibrary{fit,tikzmark}
\usetikzlibrary{positioning}

\definecolor{color1}{RGB}{228,26,28}    
\definecolor{color2}{RGB}{55,126,184}   
\definecolor{color3}{RGB}{77,175,74}    
\definecolor{color4}{RGB}{152,78,163}   
\definecolor{color5}{RGB}{255,127,0}    
\definecolor{color6}{RGB}{255,105,180} 
\definecolor{color7}{RGB}{166,86,40}    

\tikzset{
  edge1/.style={color1,thick,solid},
  edge2/.style={color2,thick,solid},
  edge3/.style={color3,thick,solid},
  edge4/.style={color4,thick,solid},
  edge5/.style={color5,thick,solid},
  edge6/.style={color6,thick,solid},
  edge7/.style={color7,thick,solid},
  fictive_edge/.style={thick, dotted},
  edge1blurred/.style={color1,thick,solid,opacity=0.3},
  edge2blurred/.style={color2,thick,solid,opacity=0.3},
  edge3blurred/.style={color3,thick,solid,opacity=0.3},
  edge4blurred/.style={color4,thick,solid,opacity=0.3},
  edge5blurred/.style={color5,thick,solid,opacity=0.3},
  edge6blurred/.style={color6,thick,solid,opacity=0.3},
  edge7blurred/.style={color7,thick,solid,opacity=0.3},
}

\tikzset{
  default/.style={circle, draw, inner sep=1pt, minimum size=5mm, font=\small},
  node1/.style={circle, draw, inner sep=1pt, minimum size=5mm, font=\small, fill=color1!40},
  node2/.style={circle, draw, inner sep=1pt, minimum size=5mm, font=\small, fill=color2!40},
  node3/.style={circle, draw, inner sep=1pt, minimum size=5mm, font=\small, fill=color3!40},
  node4/.style={circle, draw, inner sep=1pt, minimum size=5mm, font=\small, fill=color4!40},
  node5/.style={circle, draw, inner sep=1pt, minimum size=5mm, font=\small, fill=color5!40},
  node6/.style={circle, draw, inner sep=1pt, minimum size=5mm, font=\small, fill=color6!40},
}

\usepackage{tabularx,environ}
\makeatletter
\newcommand{\problemtitle}[1]{\gdef\@problemtitle{#1}}
\newcommand{\probleminput}[1]{\gdef\@probleminput{#1}}
\newcommand{\problemquestion}[1]{\gdef\@problemquestion{#1}}
\NewEnviron{problem}{
  \problemtitle{}\probleminput{}\problemquestion{}
  \BODY
  \par\addvspace{.5\baselineskip}
  \noindent
  \begin{tabularx}{\textwidth}{@{\hspace{\parindent}} l X c}
    \multicolumn{2}{@{\hspace{\parindent}}l}{\@problemtitle} \\
    \textbf{Input:} & \@probleminput \\
    \textbf{Question:} & \@problemquestion
  \end{tabularx}
  \par\addvspace{.5\baselineskip}
}
\makeatother

\makeatletter
\NewEnviron{optimproblem}{
  \problemtitle{}\probleminput{}\problemquestion{}
 \BODY
  \par\addvspace{.5\baselineskip}
  \noindent
  \begin{tabularx}{\textwidth}{@{\hspace{\parindent}} l X c}
    \multicolumn{2}{@{\hspace{\parindent}}l}{\@problemtitle} \\
    \textbf{Input:} & \@probleminput \\
    \textbf{Output:} & \@problemquestion
  \end{tabularx}
  \par\addvspace{.5\baselineskip}
}
\makeatother

\newenvironment{keywords}{
    \vspace{1ex} 
    \noindent\textbf{Keywords:} 
}{\par\vspace{1ex}} 

\RequirePackage{enumitem} 
\newlist{paraenum}{enumerate*}{1} 
\setlist[paraenum]{label=(\emph{\roman*})}  

\title{Novel neighborhood structures for incomplete round robin sports tournaments}
\author{Karel Devriesere, David Van Bulck, Dries Goossens}
\date{}

\begin{document}
\maketitle

\begin{abstract}
    The incomplete round robin sports tournament format, where each team plays the same number of games but faces only a subset of the other teams, is becoming increasingly popular in both youth and professional competitions. In contrast to conventional round robin tournaments, however, neighborhood structures for scheduling incomplete round robin tournaments have largely remained unexplored. We fill this gap by proposing two novel neighborhood structures and describe them in graph theory terms. One of them introduces a single new game followed by a minimal repair chain, while the other introduces possibly many new games but only affects a single round. The latter is shown to fully connect the solution space. We embed the neighborhoods in an adaptive late acceptance hill climbing algorithm and show that the proposed algorithm obtains high quality and new best solutions for several sets of instances from the literature, thereby empirically confirming the effectiveness of the proposed neighborhoods.
\end{abstract}


\section{Introduction}\label{sec:introduction}

The vast majority of sports timetabling problems encountered in practice have the following structure. Given is a set of teams $T=\{1,\dots,n\}$, with $n$ even, and a set of rounds $R=\{1,\dots,r\}$, such that each team has to play precisely one game in each round. Matches take place at the venue of one of the two opposing teams: the team hosting the match is the home team while the visiting team is the away team. 
A timetable designates the matches to be played in each round as well as the home team of each match \citep{nemhauser1998scheduling, henz2001scheduling}. The quality of a timetable is measured with a context-dependent objective function. For example, a common objective is to minimize the total distance that teams have to travel during the tournament \citep{ribeiro2012sports, van2020robinx}.

In this work, we will focus on incomplete round robin (iRR) tournaments, which recently have gained a lot of popularity. Structurally, an iRR tournament is similar to a single round robin tournament (RR), except that in iRR tournaments the set of available rounds is smaller than $n-1$. Hence, while in a (single) RR tournament every teams faces every opponent exactly once, in an iRR tournaments each teams faces only a subset of their possible opponents. The difference between a RR and iRR timetable is illustrated in \Cref{tab:example1}. Note that a RR timetable can always be transformed to a iRR timetable by arbitrarily taking a subset of $r$ rounds, but that the opposite is not true. For example, in \Cref{tab:example1_iRR}, teams 1, 5 and 7 face all opponents in rounds 1-5 except for each other. As it is impossible to schedule the matches $\{1,7\},\{1,5\},\{5,7\}$ in only two rounds, this set of rounds can never occur as a subset of the rounds in a RR tournament (see e.g. \citet{alexander1982premature}, \citet{schmand2022greedy}). 

\begin{table}[h!]
\centering
\caption{Difference between RR and iRR timetables with $n=8$ teams}
\label{tab:example1}
\begin{subtable}[t]{0.45\textwidth}
    \centering
\begin{tabular}{ccccccc} 
    \toprule
    R1 & R2 & R3 & R4 & R5 & R6 & R7 \\\midrule 
    3-1 & 3-2 & 7-3 & 4-3 & 3-6 & 8-3 & 3-5 \\
    7-2 & 1-4 & 2-5 & 2-1 & 2-8 & 2-4 & 6-2 \\
    8-6 & 5-8 & 1-8 & 8-7 & 5-1 & 1-6 & 7-1 \\
    4-5 & 6-7 & 6-4 & 5-6 & 7-4 & 5-7 & 4-8 \\ 
    \bottomrule
  \end{tabular}
\caption{RR timetable}
\label{tab:example1_RR}
\end{subtable}
\hspace{0.45cm}
\begin{subtable}[t]{0.45\textwidth}
\centering
\begin{threeparttable}
\begin{tabular}{ccccc} 
    \toprule
    R1 & R2 & R3 & R4 & R5 \\\midrule 
    3-1 & 4-3 & 3-2 & 7-3 & 1-6 \\
    7-2 & 2-1 & 1-4 & 2-5 & 2-8 \\
    8-6 & 8-7 & 5-8 & 1-8 & 5-3 \\
    4-5 & 5-6 & 6-7 & 6-4 & 7-4 \\ 
    \bottomrule
  \end{tabular}
\end{threeparttable}
\caption{iRR timetable with $r=5$}
\label{tab:example1_iRR}
\end{subtable}
\end{table}

It is shown by \citet{li2025beyond} and \citet{devriesere2025redesigning} that the incomplete RR format is particularly attractive for scheduling youth sports competitions. Moreover, \citet{li2025beyond} develop a metaheuristic to quickly find high quality iRR timetables. The UEFA Champions League, which is the most prestigious European club football tournament, has recently changed its format to an iRR tournament, which has further sparked the interest of the academic community (see e.g., \citet{melkonian2024integer,csato2025does,devriesere2025evaluating,Guyon2025}).


One of the most popular methods to obtain timetables is the use of local search heuristics, where the heuristic starts from a candidate solution and iteratively explores neighboring solutions based on predefined neighborhood structures. Several neighborhood structures are abundantly used to schedule sports tournaments. Although their nomenclature differs, they are commonly known as \emph{Cycle reversal} \citep{knust2006balanced}, \emph{Round Swap}, \emph{Partial Round Swap}, \emph{Team Swap}, and \emph{Partial Team Swap} \citep{ribeiro2023combinatorial}.
They have been proven successful for scheduling round robin tournaments in numerous applications (see e.g., \cite{anagnostopoulos2006simulated,gaspero2007composite,ribeiro2007heuristics,guedes2011heuristic,costa2012ils,rosati2022multi}). 
Apart from computational results, their theoretical properties and characteristics, especially with regards to connectivity, have been a central topic in the sports scheduling literature for the last few years. Informally, a neighborhood structure is called connected if, starting from a candidate timetable, any other timetable can be reached by applying a finite sequence of moves in the neighborhood. In particular, it is known that the neighborhood structures Team Swap, Partial Team Swap, Round Swap and Partial Round Swap are not connected for RR tournaments \citep{januario2016sports, urrutia2021recoloring,kashiwagi2025initial}. However, neighborhood structures tailored for iRR tournaments have remained largely unexplored. This paper aims to fill this gap. 

First, we show that if we start from a given timetable and apply RR neighborhoods, we never encounter timetables that are non-isomorphic to the initial timetable. In other words, we show that existing RR neighborhood structures are not connected for iRR tournaments. Therefore, we propose two novel neighborhoods, which we call \emph{incomplete Partial Team Swap} (\Cref{sec:i-PTS}) and \emph{incomplete Partial Round Swap} (\Cref{sec:i-PRS}). While existing neighborhoods fail to introduce new games (up to a permutation of the teams), our proposed neighborhoods are able to obtain structurally different timetables. 
More precisely, incomplete Partial Team Swap introduces a single new game followed by a minimal repair chain, while incomplete Partial Round Swap introduces possibly many new games but only affects a single round.
In fact, we show that the latter connects the solution space for certain values of $r$. This is a remarkable result, given that no connected neighborhood is known for RR tournaments. Besides a theoretical analysis of the proposed neighborhoods, we also empirically investigate their effectiveness. Using two sets of instances from the literature, we show that the proposed neighborhood structures are able to find high quality timetables when embedded in a metaheuristic framework. In particular, for several instances we find new best solutions.

The rest of the paper is organized as follows. \Cref{sec:Notation} formally describes the iRR timetabling problem. Although specific objectives and constraints differ over different problems, every iRR timetabling problem has a common set of requirements that needs to be satisfied. Next, it describes how this problem can be viewed as a graph-theoretical problem and introduces the necessary notation used in the remainder of this paper. In \Cref{sec:Neighborhoods}, we first discuss how existing neighborhood structures apply to iRR tournaments, before presenting two novel neighborhood structures. All neighborhood structures will be described in graph-theoretical terms. Next, in \Cref{sec:Connectivity} we present theoretical results concerning the connectivity of neighborhood structures, while in \Cref{sec:Experimental_results} we empirically asses the performance of the neighborhoods. Finally, \Cref{sec:Conclusion} concludes.

\section{Problem description and notation}\label{sec:Notation}

We begin this section by formally defining the iRR timetabling problem. We then show that this problem can be viewed as a graph problem. Next, we make a distinction between different types of cycles in the constructed graph, which play an important role in the remainder of this paper. We end this section with introducing definitions of different types of solution space connectivity.

\subsection{The incomplete round robin timetabling problem}

With the ordered tuple $(i,j)$, we refer to the match between $i$ and $j$ where $i$ is the home and $j$ is the away team. In the remainder of this paper, if the home team of a match is irrelevant or not yet decided, we refer to the match between $i$ and $j$ as the unordered pair $\{i,j\}$.

An iRR timetable specifies which matches are played in which rounds. More formally, an iRR timetable can be defined by the mappings $o: T \times R \rightarrow T$ and $h: T \times R \rightarrow \{0,1\}$, such that $o(i,s)$ gives the opponent of team $i$ in round $s$ and $h(i,s)$ is 1 if $i$ plays home in round $s$ and is 0 otherwise. We refer to $h(i,s)$ as the home-away status of team $i$ in round $s$. Hence, in an iRR timetable, for each pair of distinct teams $i$ and $j$ and each round $s$, it must be the case that if $o(i,s)=j$ then $o(j,s)=i$ and $h(i,s) \neq h(j,s)$.

In most practical settings, a timetable for a sports tournament must adhere to a myriad of constraints. For an overview of the most common requirements, we refer to \citet{van2020robinx}. However, at its core, the iRR timetabling problem can be defined as follows:

\begin{description}[align=left]
    \item[\textbf{Incomplete Round Robin Timetabling Problem}]
    \item[\textbf{Input:}] A set of $n$ teams $T$, with $n$ even, and a set of $r$ rounds $R$, with $r < n-1$.
    \item[\textbf{Output:}] An iRR timetable minimizing a context-depended objective function such that:
    \begin{description}
        \item[\textbf{C1}] Each team faces every opponent at most once,
        \item[\textbf{C2}] Each team has exactly one game in each round.
        \item[\textbf{C3}] Each team has at least $\lfloor \frac{r}{2} \rfloor$ and at most $\lceil \frac{r}{2} \rceil$ home games.
    \end{description}
\end{description}

An iRR timetable satisfying \textbf{C3} is called balanced and is called feasible if it satisfies \textbf{C1}-\textbf{C3}. Note that we do not specify an objective function since this depends on the application that is considered. 

\subsection{iRR timetabling as a graph problem}

It is well known that graph theory is a useful tool for modeling sports timetabling problems \citep{Drexl2007, Januario2016}. Indeed, if we let $V$ be a set of vertices such that $V = T$ and $E$ be a set of edges such that $E = \{\{i,j\}: i,j \in T: i < j\}$, then the graph $G(V,E)$ is equal to the complete graph $K_n$ and represents the set of all possible matches between teams in an iRR tournament. A perfect matching or one-factor of a graph $G$ is a set of edges $F \subseteq E$ such that each vertex is incident to precisely one edge in $F$, and hence each vertex is adjacent to one other vertex. A one-factorization of an $r$-regular graph $G(V,E)$ is a partition $\mathcal{F}$ of $E$ into disjoint perfect matchings $\mathcal{F} = \{F_1, F_2, \dots, F_{r}\}$. Since in an iRR timetable every team plays exactly once in each round (\textbf{C1}), an assignment of matches to a round corresponds to a perfect matching in $G$. Since there are $r$ rounds, and because a team can face an opponent at most once over the rounds (\textbf{C2}), the perfect matchings need to be edge disjoint, i.e.\ the set of all matches included in the timetable is an $r$-regular graph, and a one-factorization of this graph designates in which rounds which matches are played. The problem of finding $r$ edge disjoint perfect matchings in an arbitrary graph is studied in \citet{fomin2024diverse} and \citet{fomin2024diverse2}.

Next, consider a set of colors $C$, with $|C|=|R|$, and let every round correspond to a unique color. We can now color the edges of each perfect matching with the color of the corresponding round. For an example, see \Cref{fig:Equivalence_iRR}. In the remainder of this text, we use $R$ and $C$ interchangeably, and may refer to round $s \in R$ as color $s$.
Consider a mapping $c: E \rightarrow C \cup \{-1\}$ that assigns to each edge in $E$ either a color from the set $C$, or $-1$, in which case the edge remains uncolored. We define the concept of a proper $r$-regular (partial) edge coloring:

\begin{definition}[\textbf{Proper $r$-regular (partial) edge coloring}]
  Given is a graph $G(V,E)$ and a set of colors $C$. A mapping $c: E \rightarrow C \cup \{-1\}$ is called a proper $r$-regular (partial) edge coloring if:
  \begin{enumerate}
      \item For every vertex $v \in V$, the number of edges $e$ incident to $v$ such that $c(e) \in C$ is exactly $r$.
      \item For any two distinct edges $e_1, e_2$ incident to the same vertex $v$, if $c(e_1) \neq -1$ and $c(e_2) \neq -1$, then $c(e_1) \neq c(e_2)$.
  \end{enumerate}
\end{definition}

Proper $r$-regular (partial) edge colorings are related to the Maximum $r$-Edge-Colorable Subgraph Problem, which asks to maximize the number of edges in a proper edge coloring of an arbitrary graph, by using at most $r$ colors, see e.g. \citet{feige2002approximating, sanders2008asymptotic, kosowski2009approximating} and \citet{rizzi2009approximating}. In the remainder of the text, for ease of notation we may refer to a proper $r$-regular (partial) edge coloring simply as a coloring. 

A graph $G(V,E)$ whose edges are colored according to a coloring $c$ can be referred to as $G(V,E,c)$. Given a coloring $c$, define $E_c = \{e \in E: c(e) \in C\}$ and $N_c(i) = \{j \in V: c(\{i,j\}) \in C\}$. Since in a coloring of $G$ an edge is colored with at most one color and the subgraph of edges colored with every color $s \in R$ is a perfect matching, a coloring is equivalent to an iRR timetable satisfying \textbf{C1} and \textbf{C2}. Therefore, $o(i,s) = j \iff c(\{i,j\}) = s$.

A digraph is a graph whose edges are directed; a directed edge is also called an arc. In a partially colored graph, only the colored edges are directed, as directing an uncolored edge has no meaning. In every selected match $\{i,j\}$ of an iRR tournament, either $i$ or $j$ is the home team. If $i$ is the home team, we can point the edge from $j$ to $i$ and denote it as $(i,j)$, and if $j$ is the home team we can point the edge from $i$ to $j$ and write $(j,i)$. The direction of edges is defined by $h$, i.e.\ $h(i,s) = 1$ implies that there is exactly one $j \in T \setminus \{i\}$ for which $c(\{i,j\}) = s$ and $j$ points to $i$. Similarly, $h(i,s) = 0$ implies that there is exactly one $j \in T \setminus \{i\}$ for which $c(\{i,j\}) = s$ and $i$ points to $j$. If $c(\{i,j\}) = s$, then $h(i,s) \neq h(j,s)$. We are now ready to define the \emph{match graph}:

\begin{definition}[\textbf{Match graph}]
    Given a proper partial $r$-edge colored graph $G(V,E,c)$ and an orientation $h$, we refer to the oriented graph $G_c(V,E_c,c)$, or more simply $G_c$, as the match graph.
\end{definition}

The match graph represents the subset of matches that is chosen to be in the timetable; for each team, the orientation $h$ of its edges corresponds to its so-called home-away pattern \citep{van2020robinx}. It is easy to see that a coloring of $G$ is equivalent to a one-factorization of the subgraph $G_c(V,E_c,c)$.

The number of incoming edges for a vertex $v$ is given by $d^+(v)$ and is called the indegree of $v$. The number of outgoing edges is given by $d^-(v)$ and is called the outdegree of $v$. It is a well known fact that for every directed graph $G(V,E)$, it holds that $\sum_{v \in V}d^+(v) = \sum_{v \in V}d^-(v)$ = $|E|$. Let $\Delta(v) = d^+(v) - d^-(v)$. We call an $r$-regular digraph balanced if and only if $\Delta(v) = 0$ for every vertex $v$ in case $r$ is even and $\Delta(v) \in \{-1,1\}$ for every vertex $v$ in case $r$ is odd. The “unbalancedness” of a digraph can be measured by $\delta = \sum_{v:|\Delta(v)|>0}|\Delta(v)|$ if $r$ is even and $\delta = \sum_{v:|\Delta(v)|>1}|\Delta(v)|$ if $r$ is odd. Then, a digraph as a whole is called balanced if and only if $\delta = 0$. In any digraph, we can partition the set of vertices $V$ into $V^0 = \{v \in V: \Delta(v) = 0\}$, $V^+ = \{v \in V: \Delta(v) > 0\}$, and $V^- = \{v \in V: \Delta(v) < 0\}$. Note that $V^0 = \emptyset$ in case $r$ is odd.

\begin{figure}[h!]
\centering
\caption{Equivalence between a set of disjoint perfect matchings and an iRR timetable. Dashed edges correspond to unscheduled matches.}
\label{fig:Equivalence_iRR}
\begin{subfigure}[t]{0.45\textwidth}
\centering
\vspace{-10pt}
\begin{tikzpicture}
    \foreach \i/\name\label in {0/3, 1/2, 2/1, 3/8, 4/7, 5/6, 6/5, 7/4} {
        \node[default] (\name) at (\i*360/8:2cm) {\name};
    }
\draw[<-,edge2] (3) -- (2);
\draw[<-,edge1] (3) -- (1);
\draw[fictive_edge] (3) -- (8);
\draw[->,edge3] (3) -- (7);
\draw[fictive_edge] (3) -- (6);
\draw[->,edge5] (3) -- (5);
\draw[->,edge4] (3) -- (4);

\draw[<-,edge4] (2) -- (1); 
\draw[<-,edge5] (2) -- (8);
\draw[->,edge1]  (2) -- (7);
\draw[fictive_edge] (2) -- (6);
\draw[<-,edge3] (2) -- (5);
\draw[fictive_edge] (2) -- (4);

\draw[<-,edge3] (1) -- (8);
\draw[fictive_edge] (1) -- (7);
\draw[<-,edge5] (1) -- (6);
\draw[fictive_edge] (1) -- (5);
\draw[<-,edge2] (1) -- (4);

\draw[<-,edge4] (8) -- (7);
\draw[<-,edge1] (8) -- (6);
\draw[->,edge2] (8) -- (5);
\draw[fictive_edge] (8) -- (4);

\draw[->,edge2] (7) -- (6);
\draw[fictive_edge] (7) -- (5);
\draw[<-,edge5] (7) -- (4);

\draw[->,edge4] (6) -- (5);
\draw[<-,edge3] (6) -- (4);

\draw[->,edge1] (5) -- (4);

\end{tikzpicture}
\caption{Five disjoint oriented perfect matchings\\ \phantom{a}}
\label{fig:PerfectMatchings}
\end{subfigure}
\begin{subfigure}[t]{0.45\textwidth}
  \centering
  \vspace{27pt}
  \begin{tabular}{ccccc} 
    \toprule
    \tikzmarknode{colstart1}{R1} & \tikzmarknode{colstart4}{R2} & \tikzmarknode{colstart2}{R3} & \tikzmarknode{colstart3}{R4} & \tikzmarknode{colstart5}{R5} \\\midrule 
    3-1 & 4-3 & 3-2 & 7-3 & 1-6 \\
    7-2 & 2-1 & 1-4 & 2-5 & 2-8 \\
    8-6 & 8-7 & 5-8 & 1-8 & 5-3 \\
    \tikzmarknode{colend1}{4-5} & \tikzmarknode{colend4}{5-6} &\tikzmarknode{colend2}{6-7} & \tikzmarknode{colend3}{6-4} & \tikzmarknode{colend5}{7-4} \\ 
    \bottomrule
  \end{tabular}
    \begin{tikzpicture}[remember picture,overlay]
      \node[
        fit=(colstart1)(colend1),
        draw=color1,
        thick,
        inner sep=3pt,
        rounded corners
      ] {};
    \end{tikzpicture}
    \begin{tikzpicture}[remember picture,overlay]
      \node[
        fit=(colstart2)(colend2),
        draw=color2,
        thick,
        inner sep=3pt,
        rounded corners
      ] {};
    \end{tikzpicture}
    \begin{tikzpicture}[remember picture,overlay]
      \node[
        fit=(colstart3)(colend3),
        draw=color3,
        thick,
        inner sep=3pt,
        rounded corners
      ] {};
    \end{tikzpicture}
    \begin{tikzpicture}[remember picture,overlay]
      \node[
        fit=(colstart4)(colend4),
        draw=color4,
        thick,
        inner sep=3pt,
        rounded corners
      ] {};
    \end{tikzpicture}
    \begin{tikzpicture}[remember picture,overlay]
      \node[
        fit=(colstart5)(colend5),
        draw=color5,
        thick,
        inner sep=3pt,
        rounded corners
      ] {};
    \end{tikzpicture}
    \vspace{20pt}
\caption{Corresponding timetable to the colored graph in (a)}
\label{fig:timetable_iRR}
\end{subfigure}
\end{figure}
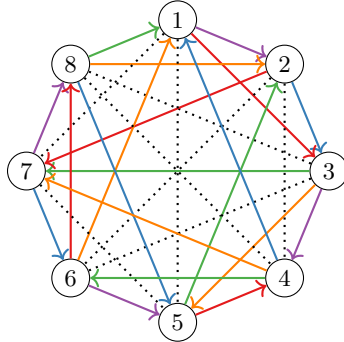

We have now introduced all the necessary tools to view the iRR timetabling problem as a graph problem. In particular, a feasible iRR timetable is equivalent to a proper partial oriented $r$-edge coloring of $K_n$ such that the resulting match graph is a balanced digraph. 


\subsection{Cycles}

A cycle in an undirected graph is defined as a list of vertices $(v_1,v_2,\dots,v_k,v_1)$ such that any two vertices $v_i, v_{i+1}$, $i = \{1,\dots,k-1\}$, are distinct and joined by the edge $\{v_{i}, v_{i+1}\}$ and $v_k$ links back to $v_1$ through the edge $\{v_{k}, v_{1}\}$. Two important types of cycles in this paper are bichromatic and alternating cycles. A bichromatic cycle is a cycle whose edges alternate between two distinct colors. An alternating cycle is a cycle whose edges alternate between uncolored edges and edges colored with one particular color, see \citet{bang1997alternating} for an overview. These are formally defined as follows:

\begin{definition}
    A $q,s$-bichromatic cycle is a list of vertices $(v_1,v_2,\dots,v_k,v_1)$ such that:
    \begin{align}
      c(\{v_i,v_{i (\!\bmod k) +1}\}) =
      \begin{cases}
         q & \text{if} \ i \text{ is odd}, \\
         s & \text{if} \ i \text{ is even}
      \end{cases}
      \quad \forall i = 1,\dots,k
    \end{align}
    for some $q,s \in R: q \neq s$.
\end{definition}

\begin{definition}
    An $s$-alternating cycle is a list of vertices $(v_1,v_2,\dots,v_k,v_1)$ such that:
    \begin{align}
      c(\{v_i,v_{i (\!\bmod k) +1}\}) =
      \begin{cases}
         s & \text{if} \ i \text{ is odd}, \\
         -1 & \text{if} \ i \text{ is even}
      \end{cases}
      \quad \forall i = 1,\dots,k
    \end{align}
    for some round $s \in R$.
\end{definition}

Note that $q,s$-bichromatic and $s$-alternating cycles exist only for $k \geq 4$ and $k$ even.

In conventional graph theory, a cycle in a directed graph is defined as a list of vertices such that any two vertices $v_i, v_{i+1}$ are joined by the arc $(v_{i}, v_{i+1})$. This implies that each vertex in the cycle has one incoming and one outgoing arc. In this paper, all colored edges are also directed. However, a vertex in a bichromatic cycle may contain two incoming or two outgoing arcs. To avoid confusion, a cycle with the property that any vertex in the cycle has exactly one incoming and one outgoing arcs is called a \emph{balanced} cycle in this paper:

\begin{definition}
    A balanced cycle is a list of vertices $(v_1,v_2,\dots,v_k,v_1)$ such that any two consecutive vertices $v_i$ and $v_{i (\!\bmod k) +1}$ are joined by the arc $(v_{i}, v_{i (\!\bmod k) +1})$. Hence, a balanced cycle only contains colored arcs.
\end{definition}

A cycle can be both balanced and bichromatic; we may refer to such a cycle as a balanced bichromatic cycle. Balanced and bichromatic paths are defined equivalently. For ease of notation, in the remainder of this paper we simply refer to arc $(v_{k}, v_{1})$ in a cycle as $(v_{k}, v_{k+1})$.

Finally, we define the concept of a balanced $s$-alternating cycle:

\begin{definition}
    A balanced $s$-alternating cycle is an alternating cycle such that, for each uncolored edge $\{i,j\}$ in the cycle, $h(i,s) \neq h(j,s)$. 
\end{definition}

In case the alternating cycle contains at least one uncolored edge $\{i,j\}$ where $h(i,s) = h(j,s)$, the alternating cycle is said to be unbalanced.



 \subsection{Connectivity in iRR}\label{sec:IntroConnectivity}

Since the input graph is $K_n$, finding a feasible iRR timetable is rather simple and can be done in linear time, see e.g.\ \cite{devriesere2026iTTP}. However, the question remains how to move to other feasible timetables. This can be done by defining neighborhood structures. Let $\mathcal{S}$ be the set of all possible timetables of either a RR or iRR tournament. Given a timetable $S\in \mathcal{S}$, a neighborhood structure $\mathcal{N}$ is a function that maps to each timetable $S$ a set of timetables $\mathcal{N}(S)$. A move in the neighborhood structure modifies $S$ to obtain $S'$ such that $S' \in \mathcal{N}(S)$. A neighborhood structure is called connected if, starting from any timetable $S \in \mathcal{S}$, any other timetable $S' \in \mathcal{S}$ can be reached by applying to $S$ a finite sequence of moves in the neighborhood (see e.g.\ \citet{nishimura2018introduction}).

Two sets of one-factors $\mathcal{F} = \{F_1, F_2, \dots, F_r\}$ and $\mathcal{F'} = \{F_1^{'}, F_2^{'}, \dots, F_r^{'}\}$ of graph $G(V,E)$ are called isomorphic if and only if there exists a bijection $\phi: V \rightarrow V$ such that, for every $k \in R$, the set $\{\{\phi(i),\phi(j)\} \ | \ i,j \in V: \{i,j\} \in F_k\} = F_l^{'}$ for some $l \in R$.  Isomorphism plays a central role in solution space connectivity. Indeed, if a neighborhood structure fails to find timetables that are non-isomorphic, then it is necessarily disconnected. We end this section by defining different types of connectivity:

\begin{definition}[\textbf{Color-wise connectivity}]
  A neighborhood structure is color-wise connected for an iRR tournament with $r$ rounds if and only if, starting from any partially $r$-edge colored graph $G(V,E,c)$, any other partially $r$-edge colored graph $G'(V,E,c')$ can be reached by applying a finite sequence of moves in the neighborhood, while remaining properly $r$-edge colored and balanced at any point in the sequence. Otherwise, we call the neighborhood color-wise disconnected.
\end{definition} 

Note that color-wise connectivity does not say anything about the orientations of the edges in the match graph (i.e.\ it is irrelevant for each match which team is the home team).

\begin{definition}[\textbf{Orientation-wise connectivity}]
  A neighborhood structure is orientation-wise connected for an iRR tournament if and only if, starting from a match graph $G_c$ with orientation $h$ such that $G_c$ is a balanced digraph, any other orientation $h'$ of $G_{c^{'}}$, such that $E_c = E_{c^{'}}$, can be reached by a finite sequence of moves in the neighborhood structure, while remaining properly $r$-edge colored and balanced at any point in the sequence. Otherwise, we call the neighborhood orientation-wise disconnected.
\end{definition}

Note that orientation-wise connectivity does not say anything about the coloring of the match graph. However, it requires that the edges in the match graphs $G_{c}$ and $G_{c'}$ are the same (i.e.\ the teams play against the same set of opponents in both cases), and that the orientations of these edges are identical.

\begin{definition}[\textbf{Full connectivity}]
  A neighborhood is fully connected for an iRR tournament if and only if it is both color-wise and orientation-wise connected.
\end{definition}

\section{Neighborhood structures}\label{sec:Neighborhoods}

In this section, we discuss the neighborhood structures. A solution is represented by a proper partially $r$-colored graph $G(V,E,c)$ (for ease of notation we will often simply refer to this graph as $G$), such that the match graph $G_c$ is a balanced digraph. We will often illustrate the neighborhood structure with an example move; unless otherwise specified, the move is applied to the graph shown in \Cref{fig:PerfectMatchings}. The neighborhood structures of Sections~\ref{sec:RS} to \ref{sec:PTS} are existing neighborhoods for RR tournaments, and we will discuss how they can be applied to iRR tournaments. Finally, in \Cref{sec:i-PTS,sec:i-PRS} two novel neighborhood structures are presented that are specifically designed for iRR tournaments. Each neighborhood move presented here can be done in linear time relative to the solution representation. 

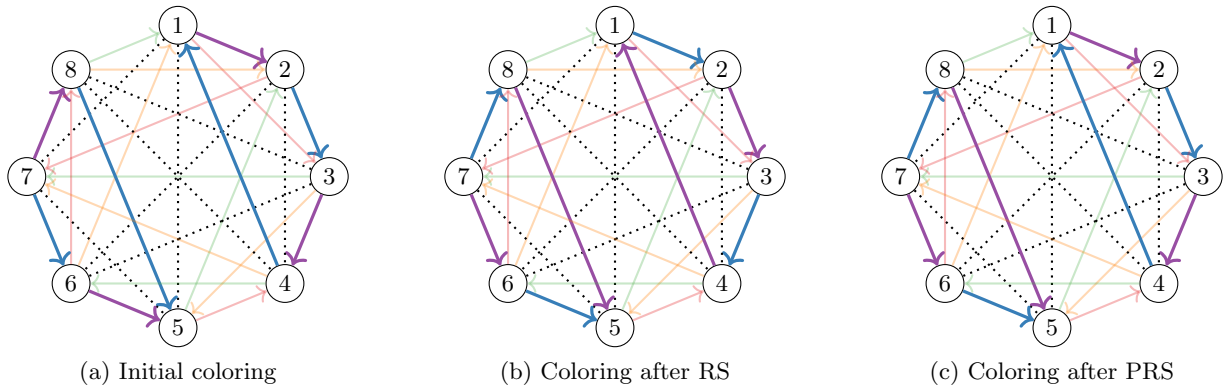
\begin{figure}[h!]
\centering
\caption{Illustration of RS and PRS with rounds R2 (purple) and R3 (blue)}
\label{fig:RS}
\begin{subfigure}{0.3\textwidth}
\centering
\begin{tikzpicture}
    \foreach \i/\name\label in {0/3, 1/2, 2/1, 3/8, 4/7, 5/6, 6/5, 7/4} {
        \node[default] (\name) at (\i*360/8:2cm) {\name};
    }
    \draw[<-,edge2,very thick] (3) -- (2);
    \draw[<-,edge1blurred] (3) -- (1);
    \draw[fictive_edge] (3) -- (8);
    \draw[->,edge3blurred] (3) -- (7);
    \draw[fictive_edge] (3) -- (6);
    \draw[->,edge5blurred] (3) -- (5);
    \draw[->,edge4,very thick] (3) -- (4);

    \draw[<-,edge4,very thick] (2) -- (1); 
    \draw[<-,edge5blurred] (2) -- (8);
    \draw[->,edge1blurred]  (2) -- (7);
    \draw[fictive_edge] (2) -- (6);
    \draw[<-,edge3blurred] (2) -- (5);
    \draw[fictive_edge] (2) -- (4);

    \draw[<-,edge3blurred] (1) -- (8);
    \draw[fictive_edge] (1) -- (7);
    \draw[<-,edge5blurred] (1) -- (6);
    \draw[fictive_edge] (1) -- (5);
    \draw[<-,edge2,very thick] (1) -- (4);

    \draw[<-,edge4,very thick] (8) -- (7);
    \draw[<-,edge1blurred] (8) -- (6);
    \draw[->,edge2,very thick] (8) -- (5);
    \draw[fictive_edge] (8) -- (4);

    \draw[->,edge2,very thick] (7) -- (6);
    \draw[fictive_edge] (7) -- (5);
    \draw[<-,edge5blurred] (7) -- (4);

    \draw[->,edge4,very thick] (6) -- (5);
    \draw[<-,edge3blurred] (6) -- (4);

    \draw[->,edge1blurred] (5) -- (4);

\end{tikzpicture}
\caption{Initial coloring}
\label{fig:RS1}
\end{subfigure}
\hfill
\begin{subfigure}{0.3\textwidth}
\centering
\begin{tikzpicture}
    \foreach \i/\name\label in {0/3, 1/2, 2/1, 3/8, 4/7, 5/6, 6/5, 7/4} {
        \node[default] (\name) at (\i*360/8:2cm) {\name};
    }
    \draw[<-,edge4,very thick] (3) -- (2);
    \draw[<-,edge1blurred] (3) -- (1);
    \draw[fictive_edge] (3) -- (8);
    \draw[->,edge3blurred] (3) -- (7);
    \draw[fictive_edge] (3) -- (6);
    \draw[->,edge5blurred] (3) -- (5);
    \draw[->,edge2,very thick] (3) -- (4);

    \draw[<-,edge2,very thick] (2) -- (1); 
    \draw[<-,edge5blurred] (2) -- (8);
    \draw[->,edge1blurred]  (2) -- (7);
    \draw[fictive_edge] (2) -- (6);
    \draw[<-,edge3blurred] (2) -- (5);
    \draw[fictive_edge] (2) -- (4);

    \draw[<-,edge3blurred] (1) -- (8);
    \draw[fictive_edge] (1) -- (7);
    \draw[<-,edge5blurred] (1) -- (6);
    \draw[fictive_edge] (1) -- (5);
    \draw[<-,edge4,very thick] (1) -- (4);

    \draw[<-,edge2,very thick] (8) -- (7);
    \draw[<-,edge1blurred] (8) -- (6);
    \draw[->,edge4,very thick] (8) -- (5);
    \draw[fictive_edge] (8) -- (4);

    \draw[->,edge4,very thick] (7) -- (6);
    \draw[fictive_edge] (7) -- (5);
    \draw[<-,edge5blurred] (7) -- (4);

    \draw[->,edge2,very thick] (6) -- (5);
    \draw[<-,edge3blurred] (6) -- (4);

    \draw[->,edge1blurred] (5) -- (4);

\end{tikzpicture}
\caption{Coloring after RS}
\label{fig:RS2}
\end{subfigure}
\hfill
\begin{subfigure}{0.3\textwidth}
\centering
\begin{tikzpicture}
    \foreach \i/\name\label in {0/3, 1/2, 2/1, 3/8, 4/7, 5/6, 6/5, 7/4} {
        \node[default] (\name) at (\i*360/8:2cm) {\name};
    }
    \draw[<-,edge2,very thick] (3) -- (2);
    \draw[<-,edge1blurred] (3) -- (1);
    \draw[fictive_edge] (3) -- (8);
    \draw[->,edge3blurred] (3) -- (7);
    \draw[fictive_edge] (3) -- (6);
    \draw[->,edge5blurred] (3) -- (5);
    \draw[->,edge4,very thick] (3) -- (4);

    \draw[<-,edge4,very thick] (2) -- (1); 
    \draw[<-,edge5blurred] (2) -- (8);
    \draw[->,edge1blurred]  (2) -- (7);
    \draw[fictive_edge] (2) -- (6);
    \draw[<-,edge3blurred] (2) -- (5);
    \draw[fictive_edge] (2) -- (4);

    \draw[<-,edge3blurred] (1) -- (8);
    \draw[fictive_edge] (1) -- (7);
    \draw[<-,edge5blurred] (1) -- (6);
    \draw[fictive_edge] (1) -- (5);
    \draw[<-,edge2,very thick] (1) -- (4);

    \draw[<-,edge2,very thick] (8) -- (7);
    \draw[<-,edge1blurred] (8) -- (6);
    \draw[->,edge4,very thick] (8) -- (5);
    \draw[fictive_edge] (8) -- (4);

    \draw[->,edge4,very thick] (7) -- (6);
    \draw[fictive_edge] (7) -- (5);
    \draw[<-,edge5blurred] (7) -- (4);

    \draw[->,edge2,very thick] (6) -- (5);
    \draw[<-,edge3blurred] (6) -- (4);

    \draw[->,edge1blurred] (5) -- (4);

\end{tikzpicture}
\caption{Coloring after PRS}
\label{fig:RS3}
\end{subfigure}
\end{figure}

\subsection{Round Swap (RS)}\label{sec:RS}

A round swap (RS) exchanges all matches between two rounds $q$ and $s$ by coloring all edges colored $q$ with $s$ and visa versa. Since there are $\frac{n}{2}$ edges in each round, a move in this neighborhood structure runs in $\mathcal{O}(|V|)$ time. Note that in both cases, the colors but not the orientations of edges are exchanged. An illustration of round swap can be found in \Cref{fig:RS1,fig:RS2}. As this neighborhood structure only involves colored edges, it can readily be applied to iRR tournaments. Since a RS merely relabels two colors, and thus only produces colorings that are isomorphic to each other, this neighborhood structure is color-wise disconnected \citep{Januario2016}. Moreover, it does not modify the direction of the edges, so it is also orientation-wise disconnected.

\subsection{Partial Round Swap (PRS)}\label{sec:PRS}

The move partial round swap (PRS) takes as input two distinct rounds $q$ and $s$ and considers the subgraph formed by all edges that are colored $q$ or $s$. Since each color corresponds to a one-factor in $G$, it follows that the subgraph formed by any two distinct colors results in a collection of edge-disjoint $q,s$-bichromatic cycles covering all vertices of $G$ (see e.g., \citet[p.78]{bondy1976graph}). PRS consists in finding one such cycle and exchanging the colors of all its edges, without changing their orientation. This requires $\mathcal{O}(|V|)$ time. For an illustration of this move, see \Cref{fig:RS1,fig:RS3}. Similarly to RS, it can readily be applied to iRR tournaments. 

In case a one-factorization of the match graph $G_c$ is perfect, any subgraph formed by two distinct colors is a Hamiltonian cycle. In this case, PRS=RS for any possible combination of two distinct rounds \citep{Januario2015}. It is well known that such one-factorization exist (for references, see \citet{Januario2016}). Therefore, PRS is both color-wise and orientation-wise disconnected.

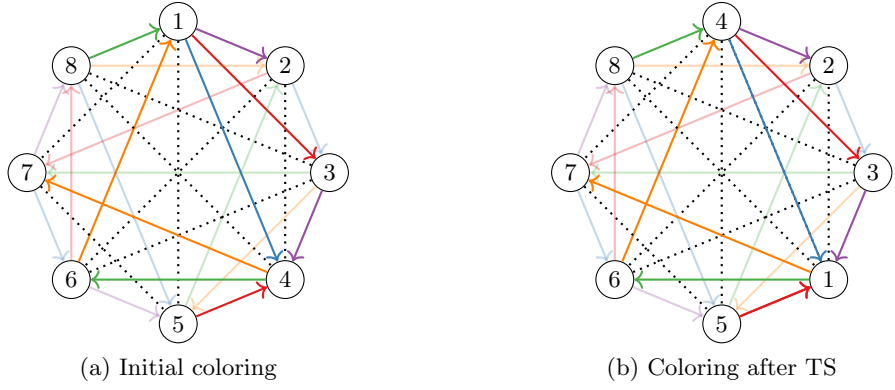
\begin{figure}[h!]
\centering
\caption{TS between teams 1 and 4}
\label{fig:TS}
\begin{subfigure}{0.30\textwidth}
\centering
\begin{tikzpicture}
    \foreach \i/\name\label in {0/3, 1/2, 2/1, 3/8, 4/7, 5/6, 6/5, 7/4} {
        \node[default] (\name) at (\i*360/8:2cm) {\name};
    }
  \draw[<-,edge2blurred] (3) -- (2);
  \draw[<-,edge1] (3) -- (1);
  \draw[fictive_edge] (3) -- (8);
  \draw[->,edge3blurred] (3) -- (7);
  \draw[fictive_edge] (3) -- (6);
  \draw[->,edge5blurred] (3) -- (5);
  \draw[->,edge4] (3) -- (4);

  \draw[<-,edge4] (2) -- (1); 
  \draw[<-,edge5blurred] (2) -- (8);
  \draw[->,edge1blurred]  (2) -- (7);
  \draw[fictive_edge] (2) -- (6);
  \draw[<-,edge3blurred] (2) -- (5);
  \draw[fictive_edge] (2) -- (4);

  \draw[<-,edge3] (1) -- (8);
  \draw[fictive_edge] (1) -- (7);
  \draw[<-,edge5] (1) -- (6);
  \draw[fictive_edge] (1) -- (5);
  \draw[<-,edge2] (4) -- (1);

  \draw[<-,edge4blurred] (8) -- (7);
  \draw[<-,edge1blurred] (8) -- (6);
  \draw[->,edge2blurred] (8) -- (5);
  \draw[fictive_edge] (8) -- (4);

  \draw[->,edge2blurred] (7) -- (6);
  \draw[fictive_edge] (7) -- (5);
  \draw[<-,edge5] (7) -- (4);

  \draw[->,edge4blurred] (6) -- (5);
  \draw[<-,edge3] (6) -- (4);

  \draw[->,edge1] (5) -- (4);

\end{tikzpicture}
\caption{Initial coloring}
\label{fig:TS1}
\end{subfigure}
\hspace{2cm}
\begin{subfigure}{0.30\textwidth}
\centering
\begin{tikzpicture}
    \foreach \i/\name\label in {0/3, 1/2, 2/4, 3/8, 4/7, 5/6, 6/5, 7/1} {
        \node[default] (\name) at (\i*360/8:2cm) {\name};
    }
  \draw[<-,edge2blurred] (3) -- (2);
  \draw[<-,edge1] (3) -- (4);
  \draw[fictive_edge] (3) -- (8);
  \draw[->,edge3blurred] (3) -- (7);
  \draw[fictive_edge] (3) -- (6);
  \draw[->,edge5blurred] (3) -- (5);
  \draw[->,edge4] (3) -- (1);

  \draw[<-,edge4] (2) -- (4);
  \draw[<-,edge5blurred] (2) -- (8);
  \draw[->,edge1blurred]  (2) -- (7);
  \draw[fictive_edge] (2) -- (6);
  \draw[<-,edge3blurred] (2) -- (5);
  \draw[fictive_edge] (1) -- (4);
  \draw[fictive_edge] (1) -- (2);

  \draw[<-,edge3] (4) -- (8);
  \draw[fictive_edge] (4) -- (7);
  \draw[<-,edge5] (4) -- (6);
  \draw[fictive_edge] (4) -- (5);
  \draw[<-,edge2] (1) -- (4);

  \draw[<-,edge4blurred] (8) -- (7);
  \draw[<-,edge1blurred] (8) -- (6);
  \draw[->,edge2blurred] (8) -- (5);
  \draw[fictive_edge] (8) -- (1);

  \draw[->,edge2blurred] (7) -- (6);
  \draw[fictive_edge] (7) -- (5);
  \draw[<-,edge5] (7) -- (1);

  \draw[->,edge4blurred] (6) -- (5);
  \draw[<-,edge3] (6) -- (1);

  \draw[->,edge1] (5) -- (1);
  \draw[->,edge1] (5) -- (1);

\end{tikzpicture}
\caption{Coloring after TS}
\label{fig:TS2}
\end{subfigure}
\end{figure}

\subsection{Team Swap (TS)}\label{sec:TS}


The move team swap (TS) takes as input two distinct teams $i$ and $j$ and exchanges their positions over the whole timetable by swapping the labels of $i$ and $j$ in $G$. It therefore runs in $\mathcal{O}(|V|)$ time. 

In contrast to RS and PRS, it can modify the set of edges that are colored and hence may seem to introduce new games. For example, after swapping teams 1 and 4 in \Cref{fig:TS2}, team 1 now faces teams 5 and 7 instead of teams 2 and 8. However, since TS merely relabels the vertices of the graph, the resulting (directed) match graph is isomorphic to the initial match graph.
Therefore, TS is both color-wise and orientation-wise disconnected.

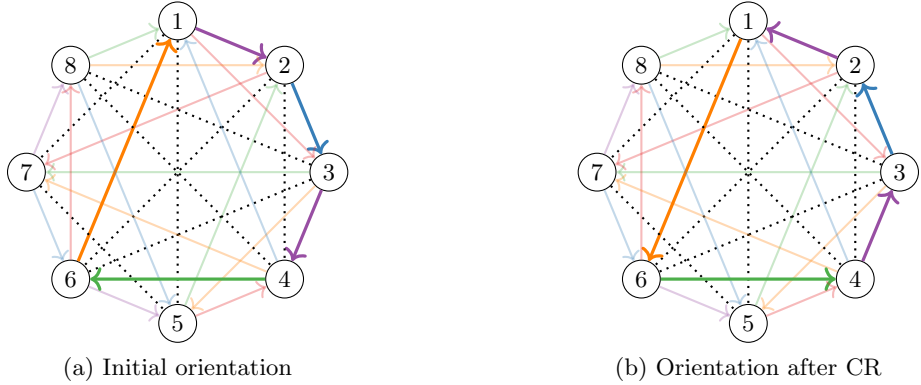
\begin{figure}[h!]
\centering
\caption{CR of (1,2,3,4,6,1)}
\label{fig:CR}
\begin{subfigure}{0.45\textwidth}
\centering
\begin{tikzpicture}
    \foreach \i/\name\label in {0/3, 1/2, 2/1, 3/8, 4/7, 5/6, 6/5, 7/4} {
        \node[default] (\name) at (\i*360/8:2cm) {\name};
    }
    \draw[<-,edge2, very thick] (3) -- (2);
    \draw[<-,edge1blurred] (3) -- (1);
    \draw[fictive_edge] (3) -- (8);
    \draw[->,edge3blurred] (3) -- (7);
    \draw[fictive_edge] (3) -- (6);
    \draw[->,edge5blurred] (3) -- (5);
    \draw[->,edge4, very thick] (3) -- (4);

    \draw[<-,edge4, very thick] (2) -- (1); 
    \draw[<-,edge5blurred] (2) -- (8);
    \draw[->,edge1blurred]  (2) -- (7);
    \draw[fictive_edge] (2) -- (6);
    \draw[<-,edge3blurred] (2) -- (5);
    \draw[fictive_edge] (2) -- (4);

    \draw[<-,edge3blurred] (1) -- (8);
    \draw[fictive_edge] (1) -- (7);
    \draw[<-,edge5, very thick] (1) -- (6);
    \draw[fictive_edge] (1) -- (5);
    \draw[<-,edge2blurred] (1) -- (4);

    \draw[<-,edge4blurred] (8) -- (7);
    \draw[<-,edge1blurred] (8) -- (6);
    \draw[->,edge2blurred] (8) -- (5);
    \draw[fictive_edge] (8) -- (4);

    \draw[->,edge2blurred] (7) -- (6);
    \draw[fictive_edge] (7) -- (5);
    \draw[<-,edge5blurred] (7) -- (4);

    \draw[->,edge4blurred] (6) -- (5);
    \draw[<-,edge3, very thick] (6) -- (4);

    \draw[->,edge1blurred] (5) -- (4);

\end{tikzpicture}
\caption{Initial orientation}
\label{fig:CR1}
\end{subfigure}
\begin{subfigure}{0.45\textwidth}
\centering
\begin{tikzpicture}
    \foreach \i/\name\label in {0/3, 1/2, 2/1, 3/8, 4/7, 5/6, 6/5, 7/4} {
        \node[default] (\name) at (\i*360/8:2cm) {\name};
    }
    \draw[->,edge2, very thick] (3) -- (2);
    \draw[<-,edge1blurred] (3) -- (1);
    \draw[fictive_edge] (3) -- (8);
    \draw[->,edge3blurred] (3) -- (7);
    \draw[fictive_edge] (3) -- (6);
    \draw[->,edge5blurred] (3) -- (5);
    \draw[<-,edge4, very thick] (3) -- (4);

    \draw[->,edge4, very thick] (2) -- (1); 
    \draw[<-,edge5blurred] (2) -- (8);
    \draw[->,edge1blurred]  (2) -- (7);
    \draw[fictive_edge] (2) -- (6);
    \draw[<-,edge3blurred] (2) -- (5);
    \draw[fictive_edge] (2) -- (4);

    \draw[<-,edge3blurred] (1) -- (8);
    \draw[fictive_edge] (1) -- (7);
    \draw[->,edge5, very thick] (1) -- (6);
    \draw[fictive_edge] (1) -- (5);
    \draw[<-,edge2blurred] (1) -- (4);

    \draw[<-,edge4blurred] (8) -- (7);
    \draw[<-,edge1blurred] (8) -- (6);
    \draw[->,edge2blurred] (8) -- (5);
    \draw[fictive_edge] (8) -- (4);

    \draw[->,edge2blurred] (7) -- (6);
    \draw[fictive_edge] (7) -- (5);
    \draw[<-,edge5blurred] (7) -- (4);

    \draw[->,edge4blurred] (6) -- (5);
    \draw[->,edge3, very thick] (6) -- (4);

    \draw[->,edge1blurred] (5) -- (4);

\end{tikzpicture}
\caption{Orientation after CR}
\label{fig:CR2}
\end{subfigure}
\end{figure}

\subsection{Cycle and Path Reversal (CR and PR)}\label{sec:PR}



The move cycle reversal (CR), first proposed by \citet{knust2006balanced} for RR tournaments, takes as input a balanced cycle and reverses its edges. For an illustration of this neighborhood, see \Cref{fig:CR}. Reversing the orientation of a balanced cycle in $G$ does not alter the number of home and away games per team. Therefore, a balanced timetable remains balanced under CR. If the match graph $G_c$ is balanced, then every vertex must have an outdegree of at least 1. Hence, we observe that for iRR tournaments with at least two rounds, an arbitrary balanced cycle always exists and can be found by a random walk, which runs in $\mathcal{O}(|V|)$ time (see \citet{chvatal1983short}). Indeed, after visiting at most $n$ vertices, we must necessarily return to an already visited vertex, at which point we have found a balanced cycle. 



We now show that if a timetable is unbalanced, it can always be made balanced by moves in the neighborhood structure path reversal (PR). A move in this neighborhood reverses the direction of a path in $G$ between two distinct nodes $i$ and $j$. First, we prove the following lemmas:

\begin{lemma}\label{lemma:PR1}
  In an unbalanced match graph $G_c(V,E_c,c)$ there always exists a path from any $w \in V^-$ to some $v \in V^+$. Similarly, for any vertex $v \in V^+$, there always exists a vertex $w \in V^-$ such that there is a path from $w$ to $v$.
\end{lemma}

\begin{proof}
  Let $w \in V^-$. Start a random walk from $w$, choosing each edge at most once until the path can no longer be extended. The walk will never end in a vertex $u \in V^0 \cup V^-$, since if we arrive at $u$ through an incoming edge, then we must also be able to leave $u$ through an unvisited outgoing edge. Hence, it follows that the walk must necessarily end in a vertex $v \in V^+$. From this walk, we extract a path by removing the balanced cycles from the walk. In other words, if the walk contains the consecutive vertices $v_0, v_1, \dots, v_0$, we remove the vertices $v_1, \dots, v_0$ and arcs in between. We do this until all vertices in the walk are unique; at this point we have retrieved a path from $w$ to $v$. 

  For the second part of the lemma, let $v \in V^+$, but now travel backwards: for vertex $v_i$, pick a vertex $v_{i+1}$ such that the edge points from $v_{i+1}$ to $v_i$, with $v_0 = v$. The walk can only end in a vertex $w \in V^-$, and from the walk we can extract a path from $w$ to $v$.
\end{proof}

Note that the proof of Lemma~\ref{lemma:PR1} involves a constructive algorithm for finding a path between two vertices. This algorithm can be implemented by, for example, Depth First Search (DFS), which runs in $\mathcal{O}(|V|+|E_c|)$ time on the match graph $G(V,E_c,c)$. 

We now show that any unbalanced iRR timetable can be made balanced by moves in the neighborhood PR. In order to do this, we first prove the following lemma:

\begin{lemma}\label{lemma:PR2}
  In any match graph $G(V,E_c,c)$, if $r$ is even, then $\delta = 4k$ for some positive integer $k$.
\end{lemma}

\begin{proof}
  Recall that $\delta$ is the total unbalancedness. Since $r$ is even, $\Delta(v) = |d^+(v) - d^-(v)|$ is even for all $v \in V$. Since $\sum_{v \in V}d^+(v) = \sum_{v \in V}d^-(v)$ for any digraph, we must have that $\sum_{v \in V: \Delta(v) \geq 1}|\Delta(v)| = \sum_{v \in V: \Delta(v) \leq -1}|\Delta(v)|$, and hence $\delta = \sum_{v \in V: \Delta(v) \geq 1}|\Delta(v)| + \sum_{v \in V: \Delta(v) \leq -1}|\Delta(v)| = 2 \left( \sum_{v \in V: \Delta(v) \geq 1}|\Delta(v)| \right)$. Therefore, $\delta$ is a multiple of 4.
\end{proof}

We are now ready to prove the following:

\begin{theorem}\label{thm:PR}
  Any unbalanced iRR timetable can be made balanced by a finite number of path reversals.
\end{theorem}

\begin{proof}
    We split the proof into two cases:
    
  \emph{Case 1: $r$ is even.}\\
  If the iRR timetable is unbalanced, then $|V^+|, |V^-| \geq 1$.
  By Lemma \ref{lemma:PR1}, there exists a path from some $w \in V^-$ to any $v \in V^+$ in the directed match graph. Reversing this path increases $\Delta(w)$ and decreases $\Delta(v)$, both by 2, and leaves $\Delta(u)$ unchanged for any node $u$ in the path different from $v$ and $w$, so the imbalance $\delta$ decreases by 4. By Lemma \ref{lemma:PR2}, $\delta$ is a multiple of 4, so if we keep reversing paths between teams with excessive home and away games, eventually $\delta = 0$. 

  \emph{Case 2: $r$ is odd.}\\
  If the iRR timetable is unbalanced, then there must exist at least one team $v$ such that $\Delta(v) \geq 2$ or $\Delta(v) \leq -2$. By Lemma \ref{lemma:PR1}, if $\Delta(v) \geq 2$ we can find a path to a vertex $w$ such that $\Delta(w) \leq -1$, and if $\Delta(v) \leq -2$, we can find a path to a vertex $u$ such that $\Delta(u) \geq 1$. In both cases, a path can be found such that reversing this path decreases the imbalance by either 1 or 2. Hence, if we keep reversing paths involving at least one team with an excessive home-or away game, eventually $\delta = 0$.
\end{proof}

An arbitrary path in $G_c$ can also be found in polynomial time by, for example, DFS. Since CR and PR only affect the orientation of the match graph, they are color-wise disconnected. \citet{knust2006balanced} show that paths of length at most 2 edges are sufficient to restore balancedness for RR tournaments. In contrast, in iRR tournaments, paths of length $\frac{n}{2}$ may be required to restore balancedness. Consider for example an iRR tournament such that the match graph is a Hamiltonian cycle, all teams except for two distinct teams $v$ and $w$ are balanced, and $v$ has
two home games and $w$ has two away games. Then, there are two paths from $v$ to $w$, and the other teams
could be placed such that both these paths consist of $\frac{n}{2}$ edges.

In \Cref{sec:CR_PR_theory}, we show that the neighborhood structure CR is orientation-wise connected. 


\subsection{Partial Team Swap (PTS)}\label{sec:PTS} 

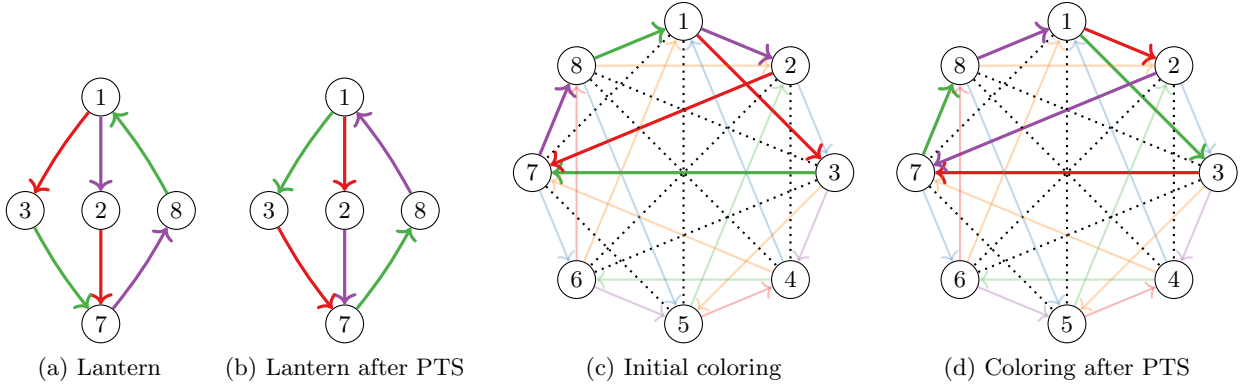
\begin{figure}[h!]
\centering
\caption{PTS between teams 1 and 7 and rounds R1 (red), R2 (purple) and R4 (green)}
\label{fig:PTS}
\begin{subfigure}{0.15\textwidth}
\centering
\begin{tikzpicture}
    
    \node[default] (1) at (0,1.5) {1};
    \node[default] (2) at (0,0) {2};
    \node[default] (3) at (-1,0) {3};
    \node[default] (8) at (1,0) {8};
    \node[default] (7) at (0,-1.5) {7};

    \draw[<-,edge3, very thick] (1) to[bend left=5] (8);
    \draw[->,edge4, very thick] (1) to (2);
    \draw[->,edge1, very thick] (1) to[bend right=5]  (3);

    \draw[->,edge4, very thick] (7) to[bend right=5]  (8);
    \draw[<-,edge1, very thick] (7) to (2);
    \draw[<-,edge3, very thick] (7) to[bend left=5]  (3);

\end{tikzpicture}
\caption{Lantern}
\label{fig:lantern}
\end{subfigure}
\hfill
\begin{subfigure}{0.20\textwidth}
\centering
\begin{tikzpicture}
    
    \node[default] (1) at (0,1.5) {1};
    \node[default] (2) at (0,0) {2};
    \node[default] (3) at (-1,0) {3};
    \node[default] (8) at (1,0) {8};
    \node[default] (7) at (0,-1.5) {7};

    \draw[<-,edge4, very thick] (1) to[bend left=5] (8);
    \draw[->,edge1, very thick] (1) to (2);
    \draw[->,edge3, very thick] (1) to[bend right=5]  (3);

    \draw[->,edge3, very thick] (7) to[bend right=5]  (8);
    \draw[<-,edge4, very thick] (7) to (2);
    \draw[<-,edge1, very thick] (7) to[bend left=5]  (3);

\end{tikzpicture}
\caption{Lantern after PTS}
\label{fig:lantern_PTS}
\end{subfigure}
\hfill
\begin{subfigure}{0.3\textwidth}
\centering
\begin{tikzpicture}
    \foreach \i/\name\label in {0/3, 1/2, 2/1, 3/8, 4/7, 5/6, 6/5, 7/4} {
        \node[default] (\name) at (\i*360/8:2cm) {\name};
    }
    \draw[<-,edge2blurred] (3) -- (2);
    \draw[<-,edge1, very thick] (3) -- (1);
    \draw[fictive_edge] (3) -- (8);
    \draw[->,edge3, very thick] (3) -- (7);
    \draw[fictive_edge] (3) -- (6);
    \draw[->,edge5blurred] (3) -- (5);
    \draw[->,edge4blurred] (3) -- (4);

    \draw[<-,edge4, very thick] (2) -- (1); 
    \draw[<-,edge5blurred] (2) -- (8);
    \draw[->,edge1, very thick]  (2) -- (7);
    \draw[fictive_edge] (2) -- (6);
    \draw[<-,edge3blurred] (2) -- (5);
    \draw[fictive_edge] (2) -- (4);

    \draw[<-,edge3, very thick] (1) -- (8);
    \draw[fictive_edge] (1) -- (7);
    \draw[<-,edge5blurred] (1) -- (6);
    \draw[fictive_edge] (1) -- (5);
    \draw[<-,edge2blurred] (1) -- (4);

    \draw[<-,edge4, very thick] (8) -- (7);
    \draw[<-,edge1blurred] (8) -- (6);
    \draw[->,edge2blurred] (8) -- (5);
    \draw[fictive_edge] (8) -- (4);

    \draw[->,edge2blurred] (7) -- (6);
    \draw[fictive_edge] (7) -- (5);
    \draw[<-,edge5blurred] (7) -- (4);

    \draw[->,edge4blurred] (6) -- (5);
    \draw[<-,edge3blurred] (6) -- (4);

    \draw[->,edge1blurred] (5) -- (4);

\end{tikzpicture}
\caption{Initial coloring}
\label{fig:PTS1}
\end{subfigure}
\begin{subfigure}{0.3\textwidth}
\centering
\begin{tikzpicture}
    \foreach \i/\name\label in {0/3, 1/2, 2/1, 3/8, 4/7, 5/6, 6/5, 7/4} {
        \node[default] (\name) at (\i*360/8:2cm) {\name};
    }
    \draw[<-,edge2blurred] (3) -- (2);
    \draw[<-,edge3, very thick] (3) -- (1);
    \draw[fictive_edge] (3) -- (8);
    \draw[->,edge1, very thick] (3) -- (7);
    \draw[fictive_edge] (3) -- (6);
    \draw[->,edge5blurred] (3) -- (5);
    \draw[->,edge4blurred] (3) -- (4);

    \draw[<-,edge1, very thick] (2) -- (1); 
    \draw[<-,edge5blurred] (2) -- (8);
    \draw[->,edge4, very thick]  (2) -- (7);
    \draw[fictive_edge] (2) -- (6);
    \draw[<-,edge3blurred] (2) -- (5);
    \draw[fictive_edge] (2) -- (4);

    \draw[<-,edge4, very thick] (1) -- (8);
    \draw[fictive_edge] (1) -- (7);
    \draw[<-,edge5blurred] (1) -- (6);
    \draw[fictive_edge] (1) -- (5);
    \draw[<-,edge2blurred] (1) -- (4);

    \draw[<-,edge3, very thick] (8) -- (7);
    \draw[<-,edge1blurred] (8) -- (6);
    \draw[->,edge2blurred] (8) -- (5);
    \draw[fictive_edge] (8) -- (4);

    \draw[->,edge2blurred] (7) -- (6);
    \draw[fictive_edge] (7) -- (5);
    \draw[<-,edge5blurred] (7) -- (4);

    \draw[->,edge4blurred] (6) -- (5);
    \draw[<-,edge3blurred] (6) -- (4);

    \draw[->,edge1blurred] (5) -- (4);

\end{tikzpicture}
\caption{Coloring after PTS}
\label{fig:PTS2}
\end{subfigure}
\end{figure}

For any subset $R' \subseteq R$, define $N_{R'}^i = \{v \in V: c(\{i,v\}) \in R'\}$. In other words, $N^i_{R'}$ denotes the set of opponents of team $i$ in the rounds $R'$. The neighborhood structure partial team swap (PTS) takes as input two distinct teams $i$ and $j$ and a subset of colors $R' \subseteq R$, such that $N_{R'}^i = N_{R'}^j$, and exchanges the opponents of $i$ and $j$ over all rounds in $R'$. In graph theoretical terms, this corresponds to a subgraph $G'(V',E')$ with $V' = \{i,j\} \cup W$, $W = N_{R'}^i = N_{R'}^j$, and $E' = \{\{i,w\}, w \in W\} \cup \{\{j,w\}, w \in W\}$, which is known as a colorful chordless lantern in case it is inclusion-wise minimal (for more details, we refer to \citet{urrutia2021recoloring} and the references herein). An example of a colorful chordless lantern and a move in this neighborhood structure are shown in \Cref{fig:PTS}. Note that the constructed lantern consists of a collection of (possibly unbalanced) paths of length two between $i$ and $j$.



In contrast to all the other neighborhood structures that we discussed so far, there does not always exists a PTS move in an iRR tournament. Consider for example an iRR tournament with only 2 rounds and at least 6 teams, such that the match graph is a Hamiltonian cycle. It can be verified that there exists no PTS move in this case. Therefore, PTS is color-wise and orientation-wise disconnected.

\citet{januario2016new} extend PTS by introducing the neighborhood structure Team And Round Swap (TARS), which again takes as input a colorful chordless lantern between two distinct teams $i$ and $j$, but allows one bichromatic path between them to be of an even length larger than two.  A further generalization called Generalized Partial Team Swap (GTPS) is discussed in \citet{ribeiro2023combinatorial}, where any number of paths can have an even length larger than two. 
However, similarly to CR, RS, PRS and PTS, both TARS and GPTS do not introduce new games into the timetable. Moreover, it is currently unknown how to efficiently perform a move in GPTS. Therefore, we do not further consider these neighborhood structures in this paper.

\subsection{Incomplete Partial Team Swap (iPTS)}\label{sec:i-PTS}

Here, we propose our first novel neighborhood structure, which we call incomplete Partial Team Swap (iPTS). It generalizes PTS discussed in \Cref{sec:PTS} to iRR tournaments. The main goal of this neighborhood is to allow the introduction of a single new game, followed by a minimal repair chain to make the timetable feasible again. At the end of this section, we show how to combine iPTS with cycle reversals in order to make moves less disruptive with respect to the orientation of the match graph. 

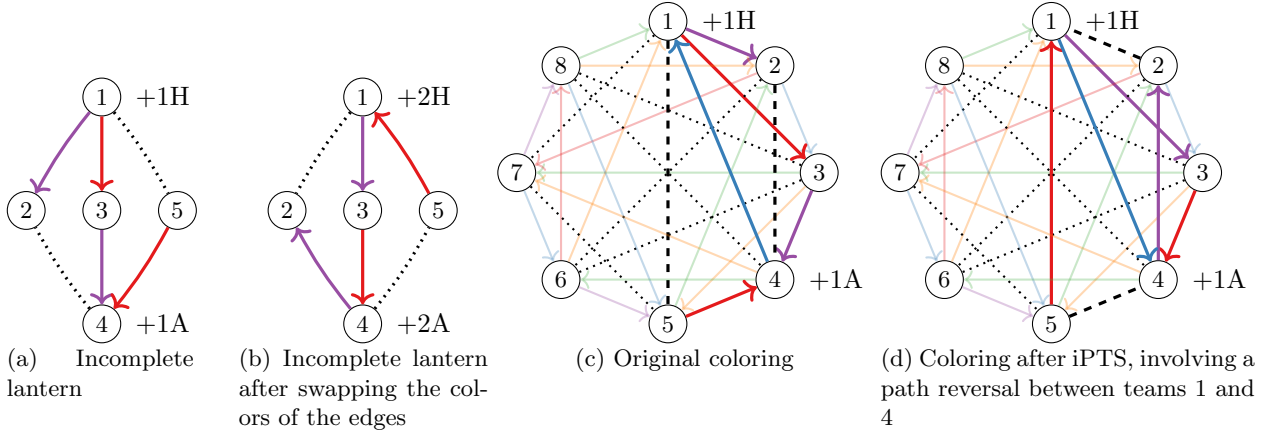
\begin{figure}[h!]
\centering
\caption{Incomplete partial team swap between teams 1 and 4}
\label{fig:PTS_iRR}
\begin{subfigure}{0.15\textwidth}
\centering
\begin{tikzpicture}
    
    \node[default] (1) at (0,1.5) {1};
    \node[default] (3) at (0,0) {3};
    \node[default] (2) at (-1,0) {2};
    \node[default] (5) at (1,0) {5};
    \node[default] (4) at (0,-1.5) {4};

    \node[right=2pt of 1] {+1H};
    \node[right=2pt of 4] {+1A};

    \draw[->,edge4, very thick] (1) to[bend right=5] (2);
    \draw[->,edge1, very thick] (1) to (3);
    \draw[fictive_edge, very thick] (1) to[bend left=5]  (5);

    \draw[fictive_edge, very thick] (4) to[bend left=5]  (2);
    \draw[<-,edge4, very thick] (4) to (3);
    \draw[<-,edge1, very thick] (4) to[bend right=5]  (5);

\end{tikzpicture}
\caption{Incomplete lantern \\ \phantom{a}}
\label{fig:PTS_iRR1}
\end{subfigure}
\hfill
\begin{subfigure}{0.20\textwidth}
\centering
\begin{tikzpicture}
    
    \node[default] (1) at (0,1.5) {1};
    \node[default] (3) at (0,0) {3};
    \node[default] (2) at (-1,0) {2};
    \node[default] (5) at (1,0) {5};
    \node[default] (4) at (0,-1.5) {4};

    \node[right=2pt of 1] {+2H};
    \node[right=2pt of 4] {+2A};

    \draw[fictive_edge, very thick] (1) to[bend right=5] (2);
    \draw[->,edge4, very thick] (1) to (3);
    \draw[<-,edge1, very thick] (1) to[bend left=5]  (5);

    \draw[->, edge4, very thick] (4) to[bend left=5]  (2);
    \draw[<-,edge1, very thick] (4) to (3);
    \draw[fictive_edge, very thick] (4) to[bend right=5]  (5);

\end{tikzpicture}
\caption{Incomplete lantern after swapping the colors of the edges}
\label{fig:PTS_iRR2}
\end{subfigure}
\begin{subfigure}{0.30\textwidth}
\centering
\begin{tikzpicture}
    \foreach \i/\name\label in {0/3, 1/2, 2/1, 3/8, 4/7, 5/6, 6/5, 7/4} {
        \node[default] (\name) at (\i*360/8:2cm) {\name};
    }

    \node[right=2pt of 1] {+1H};
    \node[right=2pt of 4] {+1A};

\draw[<-,edge2blurred] (3) -- (2);
\draw[<-,edge1, very thick] (3) -- (1);
\draw[fictive_edge] (3) -- (8);
\draw[->,edge3blurred] (3) -- (7);
\draw[fictive_edge] (3) -- (6);
\draw[->,edge5blurred] (3) -- (5);
\draw[->,edge4, very thick] (3) -- (4);

\draw[<-,edge4, very thick] (2) -- (1); 
\draw[<-,edge5blurred] (2) -- (8);
\draw[->,edge1blurred]  (2) -- (7);
\draw[fictive_edge] (2) -- (6);
\draw[<-,edge3blurred] (2) -- (5);
\draw[fictive_edge, dashed, very thick] (2) -- (4);

\draw[<-,edge3blurred] (1) -- (8);
\draw[fictive_edge] (1) -- (7);
\draw[<-,edge5blurred] (1) -- (6);
\draw[fictive_edge, dashed, very thick] (1) -- (5);
\draw[<-,edge2, very thick] (1) -- (4);

\draw[<-,edge4blurred] (8) -- (7);
\draw[<-,edge1blurred] (8) -- (6);
\draw[->,edge2blurred] (8) -- (5);
\draw[fictive_edge] (8) -- (4);

\draw[->,edge2blurred] (7) -- (6);
\draw[fictive_edge] (7) -- (5);
\draw[<-,edge5blurred] (7) -- (4);

\draw[->,edge4blurred] (6) -- (5);
\draw[<-,edge3blurred] (6) -- (4);

\draw[->,edge1, very thick] (5) -- (4);

\end{tikzpicture}
\caption{Original coloring\\ \phantom{a} \\ \phantom{a}}
\label{fig:PTS_iRR3}
\end{subfigure}
\begin{subfigure}{0.30\textwidth}
\centering
\begin{tikzpicture}
    \foreach \i/\name\label in {0/3, 1/2, 2/1, 3/8, 4/7, 5/6, 6/5, 7/4} {
        \node[default] (\name) at (\i*360/8:2cm) {\name};
    }

    \node[right=2pt of 1] {+1H};
    \node[right=2pt of 4] {+1A};
    
\draw[<-,edge2blurred] (3) -- (2);
\draw[<-,edge4, very thick] (3) -- (1);
\draw[fictive_edge] (3) -- (8);
\draw[->,edge3blurred] (3) -- (7);
\draw[fictive_edge] (3) -- (6);
\draw[->,edge5blurred] (3) -- (5);
\draw[->,edge1, very thick] (3) -- (4);

\draw[fictive_edge,dashed,very thick] (2) -- (1); 
\draw[<-,edge5blurred] (2) -- (8);
\draw[->,edge1blurred]  (2) -- (7);
\draw[fictive_edge] (2) -- (6);
\draw[<-,edge3blurred] (2) -- (5);
\draw[<-,edge4, very thick] (2) -- (4);

\draw[<-,edge3blurred] (1) -- (8);
\draw[fictive_edge] (1) -- (7);
\draw[<-,edge5blurred] (1) -- (6);
\draw[<-,edge1,very thick] (1) -- (5);
\draw[->,edge2,very thick] (1) -- (4);

\draw[<-,edge4blurred] (8) -- (7);
\draw[<-,edge1blurred] (8) -- (6);
\draw[->,edge2blurred] (8) -- (5);
\draw[fictive_edge] (8) -- (4);

\draw[->,edge2blurred] (7) -- (6);
\draw[fictive_edge] (7) -- (5);
\draw[<-,edge5blurred] (7) -- (4);

\draw[->,edge4blurred] (6) -- (5);
\draw[<-,edge3blurred] (6) -- (4);

\draw[fictive_edge,dashed,very thick] (5) -- (4);

\end{tikzpicture}
\caption{Coloring after iPTS, involving a path reversal between teams 1 and 4}
\label{fig:PTS_iRR4}
\end{subfigure}
\end{figure}

\subsubsection{Move}

A move in iPTS takes as input two distinct teams $i$ and $j$ and a color $s \in R$, such that $s \neq c(\{i,j\})$. Based on this input, it gradually builds a subset $R' \subseteq R$, such that either $N_{R'}^i \setminus N_{R'}^j = \{w^j\}$ and $N_{R'}^j \setminus N_{R'}^i = \{w^i\}$, where $\{i,w^i\}$ and $\{j,w^j\}$ are the two uncolored edges, or $N_{R'}^{i} = N_{R'}^{j}$. Let $W = N_{R'}^{i} \cup \{w^i\} = N_{R'}^{j} \cup \{w^j\}$. In graph theoretical terms, the input in the former case corresponds to the subgraph $G'(V',E')$ with $V' = \{i,j\} \cup W$ and $E' = \{\{i,w\}, \{j,w\} \ | \ w \in W\}$, such that $c(\{i,w^i\}) = -1$, $c(\{j,w^j\}) = -1$, and $\{c(\{i,w\}  \ | \  w \in W \setminus \{w^i\}\} = \{c(\{j,w\}  \ | \  w \in W \setminus \{w^j\}\} = R'$. We call this subgraph an incomplete lantern. An illustration is given in \Cref{fig:PTS_iRR1}. In this example, $i=1, j=4, W = \{2,3,5\}, w^i=5$ and $w^j=2$. In the latter case, note that the subgraph corresponds to a colorful chordless lantern.

For a given lantern (see \Cref{subsec:lanterns} to construct one in $\mathcal{O}(|V|)$ time), for every team $w \in W$, the colors between the two edges incident to $w$ are swapped. This is illustrated in \Cref{fig:PTS_iRR1,fig:PTS_iRR2}. Since an incomplete lantern contains exactly two uncolored edges, two new games are introduced. Finally, we determine the orientation of each edge. In the classic PTS, the colors between edges are typically exchanged without modifying the orientations of the edges. In iPTS, we need to specify how to deal with the orientations when coloring initially uncolored edges, which we propose to do as follows. For a given incomplete lantern, let $c^i = c(\{i,w^j\})$ and $c^j = c(\{j,w^i\})$, both before swapping colors. Then, an iPTS move uncolors the edges $\{i,w^j\}$ and $\{j,w^i\}$ and colors edge the $\{j,w^j\}$ with $c^i$ and $\{i,w^i\}$ with $c^j$. Now, we propose to set $h(i,c^j) = h(j,c^j)$ and $h(j,c^i) = h(i,c^i)$. As a result, $w^i$ and $w^j$ are guaranteed to remain balanced, while $i$ and $j$ potentially become unbalanced. In particular, $i$ and $j$ become unbalanced if and only if $h(i,c^i) \neq h(j,c^j)$. Since an incomplete lantern can be found and reversed in $\mathcal{O}(|V|)$ time (see further) and an arbitrary path between $i$ and $j$ in the match graph can be found in $\mathcal{O}(|V|+|E_c|)$ time (e.g.\ using DFS), a move in iPTS runs in $\mathcal{O}(|V|+|E_c|)$ time. 

For example, in the lantern in \Cref{fig:PTS_iRR3}, team 1 and 4 initially have no excessive home or away games and are thus perfectly balanced. However, after swapping the colors of the edges, team 1 now has one excessive home game and team 4 now has one excessive away game. We now discuss how to restore the balancedness. From \Cref{thm:PR} we know that the imbalance can be repaired by reversing a path going from the team with an excessive away game to the team with an excessive home game. For example, reversing the orientation of the arc $(4,1)$ to $(1,4)$ restores the imbalance for 1 and 4. This is illustrated in \Cref{fig:PTS_iRR4}. Hence, in order to guarantee that the match graph remains balanced, the move iPTS is complemented with at most one path reversal.

Note that if the input results in a colorful chordless lantern, the move boils down to a PTS move.

\subsubsection{Constructing lanterns}\label{subsec:lanterns}

We now show how to construct an incomplete (or colorful chordless) lantern in $\mathcal{O}(|V|)$ time. As input we require two distinct teams $i$ and $j$ and a color $s \in R$. Let the sets $W$ and $R'$ initially be empty. Our algorithm will fill the set $W$ with vertices such that the subgraph $G(V',E')$ defined above is either a colorful chordless or an incomplete lantern. For this purpose, let $\psi^{k}_{t}$ be the $k$'th color in the lantern incident to team $t$. For any pair of distinct teams $i,j$ and color $s$, we define the following recurrence relation for team $j$:

\begin{align}
  & \psi^{1}_j = s  \label{eq:psi_j0} \\
  & \psi_j^{k} = c(\{j,o(i,\psi_j^{k-1}\}) & \text{for} \ k \geq 2 \label{eq:psi_j}
\end{align}

and similarly for team $i$:

\begin{align}
  & \psi^{1}_i = s  \label{eq:psi_i0} \\
  & \psi_i^{p} = c(\{i,o(j,\psi_i^{p-1}\}) & \text{for} \ p \geq 2 \label{eq:psi_i}
\end{align}

The algorithm starts by adding to $R'$ the set of colors defined by the recurrence relation given in Equations~\eqref{eq:psi_j0} and \eqref{eq:psi_j}. First, we initialize $k$ to 1. Next, we increment $k$ by 1 until $\psi_j^k = -1$ or $\psi_j^k = s$. In case that $\psi_j^k = s$, let $E' = \{\{i,o(i,\psi^{l}_j)\}, \{j,o(j,\psi^{l+1}_j)\}: l \in \{1,\dots,k-1\}\}$ and $W = \{o(i,\psi^{l}_j): k \in \{1,\dots,k-1\}\}$. Let $V' = \{i,j\} \cup W$. Then, the subgraph $G(V',E')$ is a colorful chordless lantern. Note that a colorful chordless lantern is found in at most $r$ steps. 


In case that that $\psi_i^k = -1$, we set a new variable $p$ to 1 and add to $R'$ the set of colors defined by the recurrence relation given in Equations~\eqref{eq:psi_i0} and \eqref{eq:psi_i}. Then, we increment $p$ by 1 until $\psi_j^p = -1$. Let $E' = \{\{i,o(i,\psi^{l}_j)\}, \{j,o(i,\psi^{l}_j)\}: l \in \{1,\dots,k-1\}\} \cup \{\{j,o(i,\psi^{q}_i)\}, \{i,o(i,\psi^{q}_i)\}: q \in \{1,\dots,p-1\}\}$ and $W = \{o(i,\psi^{l}_j): l \in \{1,\dots,k-1\}\} \cup \{o(j,\psi^{q}_i): q \in \{1,\dots,p-1\}\}$. Let $V' = \{i,j\} \cup W$. Then, the subgraph $G(V',E')$ is an incomplete lantern. An incomplete lantern is found in at most $r+2$ steps. Since $r < |V|-1$, both a colorful chordless and incomplete lantern can be found $\mathcal{O}(|V|)$ time.


Note that if we take as input a vertex $w$ such that there exists two distinct teams $i$ and $j$ with $c(\{i,w\}) \neq -1$ and $c(\{j,w\}) = -1$, the constructive algorithm that takes as input $i,j$ and $c(\{i,w\})$ is guaranteed to produce an incomplete lantern. In contrast, there is no trivial way of choosing vertices $i,j$ and color $c$ such that the lantern is guaranteed to result in a colorful chordless lantern (i.e.\ not containing uncolored edges).

\subsubsection{iPTS with internal cycle reversal (iPTS-CR)}

While iPTS guarantees balancedness for the teams in the middle of the lantern, it may change the match orientations for a large number of teams. However, modifying many home-away patterns may be undesirable, for example when strict constraints are imposed on them. Therefore, we now show how to combine iPTS with internal cycle reversals to mitigate this effect. We refer to the resulting neighborhood as 
iPTS-CR, which can be seen as an ejection chain of iPTS and CR moves.

Consider the lantern (colorful chordless or incomplete) with vertex set $V = \{i,j\} \cup W$ following from the iPTS move. We now partition the set $W$ into three sets $W_1,W_2$ and $W_3$. Let $W_1 = \{w \in W \ | \ c(i,w) \neq -1 \ \land \ h(w,c(i,w))=0 \ \land \ h(w,c(j,w))=1\}$, $W_2 = \{w \in W \ | \ c(i,w) \neq -1 \ \land \ h(w,c(i,w))=1 \ \land \ h(w,c(j,w))=0\}$ and $W_3 = W \setminus (W_1 \cup W_2)$. Then, for any $w_1 \in W_1$ and $w_2 \in W_2$, the subgraph in the lantern consisting of $i,j,w_1$ and $w_2$ is a balanced internal cycle. In this case, swapping the orientations guarantees that the home-away patterns of $w_1$ and $w_2$ remain the same as before the move. For example, in the colorful chordless lantern in \Cref{fig:PTS2}, an iPTS move (corresponding to a PTS move) modifies the home-away patterns of all five teams in the lantern. However, if we swap the orientations of the cycle consisting of teams 1,3,7 and 8, the orientations of all these teams are restored in the affected rounds. Therefore, moves in the neighborhood iPTS-CR consist in doing a conventional iPTS move, followed by $\min \{|W_1|, |W_2|\}$ arbitrary cycle reversals.
This way, moves in this neighborhood now have a less disruptive effect on the current solution.

\subsection{Incomplete Partial Round Swap (iPRS)}\label{sec:i-PRS}

Finally, we introduce the neighborhood structure \emph{incomplete partial round swap} (iPRS). The main goal of this neighborhood is to introduce new games into one single round $s \in R$, while leaving all other rounds unmodified. 
It does this by swapping the colored for the uncolored edges of an $s$-alternating cycle (see \Cref{sec:Notation}). 
We introduce two versions of iPRS: one that does not affect the original home-away patterns (iPRS-B) and another that does not offer this guarantee (iPRS-U). 



\subsubsection{Balanced iPRS (iPRS-B)}\label{sec:iPRS-B}

The neighborhood structure iPRS-B takes as input a round $s \in R$, finds a balanced $s$-alternating cycle, and swaps the colored for the uncolored edges. We now propose how to find such a cycle. First, we construct the bipartite graph $G_s^b(V^{h}_s, V^{a}_s, A_s^b)$, with $V^{h}_s = \{i \in V: h(i,s) = 1\}$, $V^{a}_s = \{i \in V: h(i,s) = 0\}$ and arc set $A_s^b = \{(j,i): i \in V^{h}_s, j \in V^{a}_s, c(\{i,j\}) = s\} \cup \{(i,j): i \in V^{h}_s, j \in V^{a}_s, c(\{i,j\}) = -1\}$. \Cref{fig:BipartiteMatching_iRR2} gives this graph based on the graph shown in \Cref{fig:BipartiteMatching_iRR1}, for the purple color. We now run a cycle detection algorithm on $G_s^b$. Indeed, any cycle found in $G^b_s$ must be a balanced alternating cycle. To see this, start from an arbitrary vertex and note that any arc in $A_s^b$ with color $s$ points from a vertex in $V^{a}_s$ to a vertex in $V^{h}_s$, and any uncolored arc points from a vertex in $V^{h}_s$ to a vertex in $V^{a}_s$. Hence, by construction, the cycle is alternating and balanced. 

Given a balanced $s$-alternating cycle, a move in iPRS-B consists of first swapping the colored for the uncolored edges and then setting the orientations of the newly colored edges, such that each team's home-away status in round $s$ is preserved.
For example, in \Cref{fig:BipartiteMatching_iRR2} the alternating balanced cycle with edges $\{1,2\}, \{2,6\}, \{5,6\}$ and $\{1,5\}$ can be found, and applying iPRS-B on this cycle results in the coloring shown in \Cref{fig:BipartiteMatching_iRR3}. An arbitrary cycle in $G_s^b$ can be found with DFS, which runs in $\mathcal{O}(|V|+|A_s^b|) = \mathcal{O}(|V|^2)$ time.

\begin{figure}[h!]
\centering
\caption{iPRS-B}
\label{fig:BipartiteMatching_iRR}
\begin{subfigure}{0.3\textwidth}
\centering
\begin{tikzpicture}
    \foreach \i/\name\label in {0/3, 1/2, 2/1, 3/8, 4/7, 5/6, 6/5, 7/4} {
        \node[default] (\name) at (\i*360/8:2cm) {\name};
    }
\draw[<-,edge2blurred] (3) -- (2);
\draw[<-,edge1blurred] (3) -- (1);
\draw[fictive_edge, very thick, dashed] (3) -- (8);
\draw[->,edge3blurred] (3) -- (7);
\draw[fictive_edge, very thick, dashed] (3) -- (6);
\draw[->,edge5blurred] (3) -- (5);
\draw[->,edge4, very thick] (3) -- (4);

\draw[<-,edge4, very thick] (2) -- (1); 
\draw[<-,edge5blurred] (2) -- (8);
\draw[->,edge1blurred]  (2) -- (7);
\draw[fictive_edge, very thick, dashed] (2) -- (6);
\draw[<-,edge3blurred] (2) -- (5);
\draw[fictive_edge, very thick, dashed] (2) -- (4);

\draw[<-,edge3blurred] (1) -- (8);
\draw[fictive_edge, very thick, dashed] (1) -- (7);
\draw[<-,edge5blurred] (1) -- (6);
\draw[fictive_edge, very thick, dashed] (1) -- (5);
\draw[<-,edge2blurred] (1) -- (4);

\draw[<-,edge4, very thick] (8) -- (7);
\draw[<-,edge1blurred] (8) -- (6);
\draw[->,edge2blurred] (8) -- (5);
\draw[fictive_edge, very thick, dashed] (8) -- (4);

\draw[->,edge2blurred] (7) -- (6);
\draw[fictive_edge, very thick, dashed] (7) -- (5);
\draw[<-,edge5blurred] (7) -- (4);

\draw[->,edge4, very thick] (6) -- (5);
\draw[<-,edge3blurred] (6) -- (4);

\draw[->,edge1blurred] (5) -- (4);

\end{tikzpicture}
\caption{Initial coloring\\ \phantom{a}}
\label{fig:BipartiteMatching_iRR1}
\end{subfigure}
\begin{subfigure}{0.3\textwidth}
\centering
\begin{tikzpicture}
    \node[default] (1) at (-1,  2.0) {1};
    \node[default] (3) at (-1,  0.667) {3};
    \node[default] (6) at (-1, -0.667) {6};
    \node[default] (7) at (-1, -2.0) {7};

    \node[default] (2) at (1,  2.0) {2};
    \node[default] (4) at (1,  0.667) {4};
    \node[default] (5) at (1, -0.667) {5};
    \node[default] (8) at (1, -2.0) {8};

    \node[above=2pt of 1] {$V^a_s$};
    \node[above=2pt of 2] {$V^h_s$};

    \draw[->,edge4, very thick] (1) -- (2);
    \draw[<-, fictive_edge, very thick, dashed] (1) -- (5);
    \draw[->,edge4blurred, very thick] (3) -- (4);
    \draw[<-, fictive_edge, dashed, opacity=0.3] (3) -- (8);
    \draw[<-, fictive_edge, very thick, dashed] (6) -- (2);
    \draw[->,edge4, very thick] (6) -- (5);
    \draw[<-, fictive_edge, dashed, opacity=0.3] (7) -- (5);
    \draw[->,edge4blurred, very thick] (7) -- (8);

\end{tikzpicture}
\caption{Auxiliary graph $G^b_s$ \\ \phantom{a}}
\label{fig:BipartiteMatching_iRR2}
\end{subfigure}
\begin{subfigure}{0.3\textwidth}
\centering
\begin{tikzpicture}
    \foreach \i/\name\label in {0/3, 1/2, 2/1, 3/8, 4/7, 5/6, 6/5, 7/4} {
        \node[default] (\name) at (\i*360/8:2cm) {\name};
    }
\draw[<-,edge2blurred] (3) -- (2);
\draw[<-,edge1blurred] (3) -- (1);
\draw[fictive_edge, very thick, dashed] (3) -- (8);
\draw[->,edge3blurred] (3) -- (7);
\draw[fictive_edge, very thick, dashed] (3) -- (6);
\draw[->,edge5blurred] (3) -- (5);
\draw[->,edge4, very thick] (3) -- (4);

\draw[fictive_edge, very thick, dashed] (2) -- (1); 
\draw[<-,edge5blurred] (2) -- (8);
\draw[->,edge1blurred]  (2) -- (7);
\draw[<-,edge4, very thick] (2) -- (6);
\draw[<-,edge3blurred] (2) -- (5);
\draw[fictive_edge, very thick, dashed] (2) -- (4);

\draw[<-,edge3blurred] (1) -- (8);
\draw[fictive_edge, very thick, dashed] (1) -- (7);
\draw[<-,edge5blurred] (1) -- (6);
\draw[->,edge4, very thick] (1) -- (5);
\draw[<-,edge2blurred] (1) -- (4);

\draw[<-,edge4, very thick] (8) -- (7);
\draw[<-,edge1blurred] (8) -- (6);
\draw[->,edge2blurred] (8) -- (5);
\draw[fictive_edge, very thick, dashed] (8) -- (4);

\draw[->,edge2blurred] (7) -- (6);
\draw[fictive_edge, very thick, dashed] (7) -- (5);
\draw[<-,edge5blurred] (7) -- (4);

\draw[fictive_edge, very thick, dashed] (6) -- (5);
\draw[<-,edge3blurred] (6) -- (4);

\draw[->,edge1blurred] (5) -- (4);

\end{tikzpicture}
\caption{Coloring after iPRS-B\\ \phantom{a}}
\label{fig:BipartiteMatching_iRR3}
\end{subfigure}
\end{figure}

Given a round $s \in R$, it is possible that no balanced alternating cycle in $G^b_s$ exists. \Cref{fig:i-PRS-B_no_move} shows an iRR tournament with $r=5$. Since all even teams play home and all odd teams play away, and only vertices with the same parity are connected by an uncolored edge, the graph $G^b_s$ consists only of edges colored purple. 

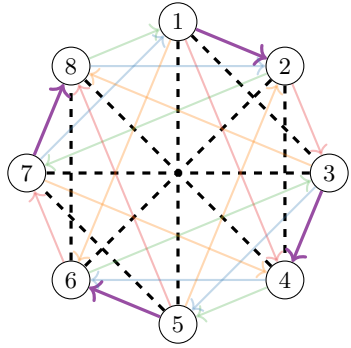
\begin{figure}[h!]
\centering
\caption{Graph with no balanced alternating cycle}
\label{fig:i-PRS-B_no_move}
\begin{tikzpicture}
    \foreach \i/\name\label in {0/3, 1/2, 2/1, 3/8, 4/7, 5/6, 6/5, 7/4} {
        \node[default] (\name) at (\i*360/8:2cm) {\name};
    }

  \draw[fictive_edge, very thick, dashed] (1) -- (3);
  \draw[fictive_edge, very thick, dashed] (2) -- (4);
  \draw[fictive_edge, very thick, dashed] (5) -- (7);
  \draw[fictive_edge, very thick, dashed] (6) -- (8);
  \draw[fictive_edge, very thick, dashed] (1) -- (5);
  \draw[fictive_edge, very thick, dashed] (3) -- (7);
  \draw[fictive_edge, very thick, dashed] (2) -- (6);
  \draw[fictive_edge, very thick, dashed] (4) -- (8);

  \draw[->,edge4, very thick] (1) -- (2);
  \draw[->,edge4, very thick] (3) -- (4);
  \draw[->,edge4, very thick] (5) -- (6);
  \draw[->,edge4, very thick] (7) -- (8);

  \draw[->,edge3blurred] (8) -- (1);
  \draw[->,edge3blurred] (2) -- (7);
  \draw[->,edge3blurred] (4) -- (5);
  \draw[->,edge3blurred] (6) -- (3);

  \draw[->,edge1blurred] (1) -- (4);
  \draw[->,edge1blurred] (6) -- (7);
  \draw[->,edge1blurred] (2) -- (3);
  \draw[->,edge1blurred] (5) -- (8);

  \draw[->,edge2blurred] (4) -- (6);
  \draw[->,edge2blurred] (7) -- (1);
  \draw[->,edge2blurred] (3) -- (5);
  \draw[->,edge2blurred] (8) -- (2);

  \draw[->,edge5blurred] (1) -- (6);
  \draw[->,edge5blurred] (7) -- (4);
  \draw[->,edge5blurred] (3) -- (8);
  \draw[->,edge5blurred] (5) -- (2);

\end{tikzpicture}
\end{figure}

\subsubsection{Unbalanced iPRS (iPRS-U)}

Rather than considering only balanced alternating cycles, iPRS-U starts by finding an alternating cycle that is not necessarily balanced and swaps its colored and uncolored edges, initially ignoring their home-away status. The balance is later restored by a series of PR moves. 

For a move in iPRS-U involving round $s \in R$, we take as input the graph $G^u_s(V, E_{s,-1})$, where $E_{s,-1} = \{e \in E: c(e) = -1 \ \text{or} \ c(e) = s\}$. This graph always contains an alternating cycle. Indeed, it follows from \citet{grossman1983alternating} that if both monochromatic subgraphs of a bichromatic graph are regular and nontrivial, the graph has an alternating cycle.

In \citet{bang1997alternating}, an algorithm is presented that finds the largest alternating cycle in $\mathcal{O}(|V|^3)$ time. Here, we present an $\mathcal{O}(|V|^2)$ algorithm for finding an arbitrary alternating cycle in $G^u_s$, which is a modification of the conventional DFS cycle detection algorithm. For each vertex $v \in V$, we define $N_{-1}^v = \{w \in V: c({v,w}) = -1\}$. We also arbitrarily order the set $N_{-1}^v$. Let $N_{-1,i}^v$ be the $i'{\text{th}}$ element in the set $N_{-1}^v$. Our algorithm labels the vertices as it searches for a cycle. For this purpose, define $l(v)$ to be the current label of vertex $v$ and initialize $l(v) = 0$ for all $v \in V$. Finally, we define $p$ such that $p(v)$ is the predecessor of $v$ in the DFS tree. Initially, we set $p(v)=\emptyset$ for each $v \in V$.

The algorithm works as follows. We start at an arbitrary vertex $u$ and set $l(u)=1$. At any vertex $v$, the next candidate $w$ is chosen based on the parity of $l(v)$: if $l(v)$ is even, we set $w$ equal to $o(v,s)$, and if $l(v)$ is odd, we select $w$ by iteratively going over all the vertices in $N_{-1}^v$. For a chosen candidate $w$, we evaluate three mutually exclusive cases:

\begin{enumerate}
  \item If $l(w) = 0$, set $l(w)=l(v)+1$, $p(w)=v$. 
  \item If $l(w) > 0$ and $l(w)$ is even, we have detected an alternating cycle. The cycle is formed by the path from $w$ to $v$, as defined by $p$, plus the edge $\{v,w\}$. 
  \item If $l(w) > 0$ and $l(w)$ is odd, this edge does not lead to an alternating cycle. If this is the case, we ignore $w$ and pick the next candidate in $N_{-1}^v$. 
\end{enumerate}

In case that all candidates for $v$ are exhausted, we return to $p(v)$ and reset $l(v) = 0$. In case that $p(v)=\emptyset$, we iteratively go over all vertices in $V$ that are currently unlabeled, and select that vertex as a new starting point $u$ until all vertices have been labeled. Since an alternating cycle is guaranteed to exist, the algorithm runs in $\mathcal{O}(|V| + |E_{s,-1}|) = \mathcal{O}(|V|+|V|^2)$ time.

The detected alternating cycle is in this case not necessarily balanced. In order to determine the orientation of the uncolored edges, we proceed as follows. First, we construct the sets of vertices $W^+$ and $W^-$. 
For any uncolored edge $\{i,j\}$ in the $s$-alternating cycle where $h(i,s) = h(j,s) = 1$, arbitrarily add $i$ to $W^-$ and orient $\{i,j\}$ as $(j,i)$.
Similarly, for any uncolored edge $\{i,j\}$ in the $s$-alternating cycle where $h(i,s) = h(j,s) = 0$, arbitrarily add $i$ to $W^+$ and orient $\{i,j\}$ as $(i,j)$. 
For example, if in Figure~\ref{fig:BipartiteMatching_iRR1} we swap the colored for the uncolored edges in the cycle with the edges $\{7,8\}, \{3,8\}, \{3,4\}, \{2,4\}, \{1,2\}, \{1,7\}$, we have to color the edge $\{2,4\}$, but both teams play home in R2 (purple). Similarly, we have to color edge $\{1,7\}$, but both teams play away in R2. Hence, the orientation of $\{1,7\}$ is set to (1,7), and $\{2,4\}$ is set to (4,2). Then, giving $\{3,8\}$ the orientation $(8,3)$, results in the graph shown in \Cref{fig:Matching_iRR1}. From \Cref{{thm:PR}}, we know that we can restore the balance by doing a finite number of path reversals. Hence, any move of iPRS-U is finished by a finite sequence of path reversals between vertices in $W^-$ and $W^+$, as is illustrated in \Cref{fig:Matching_iRR}. Note that any iPRS-B move is an iPRS-U move, while the reverse is not true. Hence we have that iPRS-B $\subseteq$ iPRS-U. Since there are at most $\frac{|v|}{4}$ vertices in $W^+$, and a path can be found in $\mathcal{O}(|V|+|E_c|) = \mathcal{O}(|V|^2)$ time in the match graph $G_c$, a move in this neighborhood runs in $\mathcal{O}(|V|^3)$ time.

\begin{figure}[h!]
\centering
\caption{iPRS-U}
\label{fig:Matching_iRR}
\begin{subfigure}{0.30\textwidth}
\centering
\begin{tikzpicture}
    \foreach \i/\name\label in {0/3, 1/2, 2/1, 3/8, 4/7, 5/6, 6/5, 7/4} {
        \node[default] (\name) at (\i*360/8:2cm) {\name};
    }
    \node[above=2pt of 1] {+1H};
    \node[above=2pt of 2] {+1A};
\draw[<-,edge2blurred] (3) -- (2);
\draw[<-,edge1blurred] (3) -- (1);
\draw[->,edge4, very thick] (3) -- (8);
\draw[->,edge3blurred] (3) -- (7);
\draw[fictive_edge, very thick, dashed] (3) -- (6);
\draw[->,edge5blurred] (3) -- (5);
\draw[fictive_edge, very thick, dashed] (3) -- (4);

\draw[fictive_edge, very thick, dashed] (2) -- (1); 
\draw[<-,edge5blurred] (2) -- (8);
\draw[->,edge1blurred]  (2) -- (7);
\draw[fictive_edge, very thick, dashed] (2) -- (6);
\draw[<-,edge3blurred] (2) -- (5);
\draw[->,edge4, very thick] (2) -- (4);

\draw[<-,edge3blurred] (1) -- (8);
\draw[<-,edge4, very thick] (1) -- (7);
\draw[<-,edge5blurred] (1) -- (6);
\draw[fictive_edge, very thick, dashed] (1) -- (5);
\draw[<-,edge2blurred] (1) -- (4);

\draw[fictive_edge, very thick, dashed] (8) -- (7);
\draw[<-,edge1blurred] (8) -- (6);
\draw[<-,edge2blurred] (8) -- (5);
\draw[fictive_edge, very thick, dashed] (8) -- (4);

\draw[->,edge2blurred] (7) -- (6);
\draw[fictive_edge, very thick, dashed] (7) -- (5);
\draw[->,edge5blurred] (7) -- (4);

\draw[->,edge4, very thick] (6) -- (5);
\draw[<-,edge3blurred] (6) -- (4);

\draw[<-,edge1blurred] (5) -- (4);

\end{tikzpicture}
\caption{Coloring after iPRS-U}
\label{fig:Matching_iRR1}
\end{subfigure}
\begin{subfigure}{0.30\textwidth}
\centering
\begin{tikzpicture}
    \foreach \i/\name\label in {0/3, 1/2, 2/1, 3/8, 4/7, 5/6, 6/5, 7/4} {
        \node[default] (\name) at (\i*360/8:2cm) {\name};
    }
    \node[above=2pt of 1] {+1H};
    \node[above=2pt of 2] {+1A};
\draw[<-,edge2blurred] (3) -- (2);
\draw[<-,edge1blurred] (3) -- (1);
\draw[->,edge4, very thick] (3) -- (8);
\draw[->,edge3blurred] (3) -- (7);
\draw[fictive_edge, very thick, dashed] (3) -- (6);
\draw[->,edge5blurred] (3) -- (5);
\draw[fictive_edge, very thick, dashed] (3) -- (4);

\draw[fictive_edge, very thick, dashed] (2) -- (1); 
\draw[<-,edge5blurred] (2) -- (8);
\draw[->,edge1, very thick]  (2) -- (7);
\draw[fictive_edge, very thick, dashed] (2) -- (6);
\draw[<-,edge3blurred] (2) -- (5);
\draw[->,edge4, very thick] (2) -- (4);

\draw[<-,edge3, very thick] (1) -- (8);
\draw[<-,edge4, very thick] (1) -- (7);
\draw[<-,edge5blurred] (1) -- (6);
\draw[fictive_edge, very thick, dashed] (1) -- (5);
\draw[<-,edge2blurred] (1) -- (4);

\draw[fictive_edge, very thick, dashed] (8) -- (7);
\draw[<-,edge1, very thick] (8) -- (6);
\draw[<-,edge2blurred] (8) -- (5);
\draw[fictive_edge, very thick, dashed] (8) -- (4);

\draw[->,edge2, very thick] (7) -- (6);
\draw[fictive_edge, very thick, dashed] (7) -- (5);
\draw[->,edge5blurred] (7) -- (4);

\draw[->,edge4, very thick] (6) -- (5);
\draw[<-,edge3blurred] (6) -- (4);

\draw[<-,edge1blurred] (5) -- (4);

\end{tikzpicture}
\caption{Path (2,7,6,8,1) from 2 to 1}
\label{fig:Matching_iRR2}
\end{subfigure}
\begin{subfigure}{0.30\textwidth}
\centering
\begin{tikzpicture}
    \foreach \i/\name\label in {0/3, 1/2, 2/1, 3/8, 4/7, 5/6, 6/5, 7/4} {
        \node[default] (\name) at (\i*360/8:2cm) {\name};
    }
    \node[above=2pt of 1] {+0H};
    \node[above=2pt of 2] {+0A};
\draw[<-,edge2blurred] (3) -- (2);
\draw[<-,edge1blurred] (3) -- (1);
\draw[->,edge4, very thick] (3) -- (8);
\draw[->,edge3blurred] (3) -- (7);
\draw[fictive_edge, very thick, dashed] (3) -- (6);
\draw[->,edge5blurred] (3) -- (5);
\draw[fictive_edge, very thick, dashed] (3) -- (4);

\draw[fictive_edge, very thick, dashed] (2) -- (1); 
\draw[<-,edge5blurred] (2) -- (8);
\draw[<-,edge1, very thick]  (2) -- (7);
\draw[fictive_edge, very thick, dashed] (2) -- (6);
\draw[<-,edge3blurred] (2) -- (5);
\draw[->,edge4, very thick] (2) -- (4);

\draw[->,edge3, very thick] (1) -- (8);
\draw[<-,edge4, very thick] (1) -- (7);
\draw[<-,edge5blurred] (1) -- (6);
\draw[fictive_edge, very thick, dashed] (1) -- (5);
\draw[<-,edge2blurred] (1) -- (4);

\draw[fictive_edge, very thick, dashed] (8) -- (7);
\draw[->,edge1, very thick] (8) -- (6);
\draw[<-,edge2blurred] (8) -- (5);
\draw[fictive_edge, very thick, dashed] (8) -- (4);

\draw[<-,edge2, very thick] (7) -- (6);
\draw[fictive_edge, very thick, dashed] (7) -- (5);
\draw[->,edge5blurred] (7) -- (4);

\draw[->,edge4, very thick] (6) -- (5);
\draw[<-,edge3blurred] (6) -- (4);

\draw[<-,edge1blurred] (5) -- (4);

\end{tikzpicture}
\caption{Orientation after path reversal}
\label{fig:Matching_iRR3}
\end{subfigure}
\end{figure}


 
\section{Connectivity results}\label{sec:Connectivity}

In this section, we present our theoretical results concerning the connectivity of the neighborhoods.


\subsection{Cycle and Path Reversal}\label{sec:CR_PR_theory}

In \citet{knust2006balanced}, it is shown that CR is orientation-wise connected for RR tournaments with $r$ even ($n$ odd), while PR is orientation-wise connected if $r$ is odd ($n$ even). We generalize this result to the iRR context in the following theorem. 

\begin{theorem}\label{thm:ConnectivityCR}
  CR is orientation-wise connected for iRR tournaments in case that $r$ is even and CR+PR is orientation-wise connected for iRR tournaments in case that $r$ is odd.
\end{theorem}

\begin{proof}
  Let $h$ and $h'$ be two balanced orientations of the same match graph $G_c$, resulting in the directed graphs $G_c(V,A)$ and $G^{'}_c(V,A')$, respectively. Since moves in CR and PR do not modify the coloring, we ignore the colors of the edges. Let $D = A \setminus A'$ represent the set of arcs in $A$ that have a different orientation in $A'$. We now consider two cases.

  \emph{Case 1: $r$ is even}. Since $h$ and $h'$ are balanced and $r$ is even, for any vertex $v \in V$ it holds that if there is an incoming arc for $v$ in $h$ that is outgoing in $h'$, then there exists another arc that is incoming for $v$ in $h'$ but outgoing in $h$. Similarly, if there is an outgoing arc for $v$ in $h$ that is incoming for $v$ in $h'$, then there must exist another arc that is outgoing for $v$ in $h'$ but incoming in $h$. Hence, any vertex that is incident to an arc in $G(V,D)$ has an outdegree of at least 1, so we can find an arbitrary cycle $C$ in $G(V,D)$ (see e.g.\ \citet{chvatal1983short}) and reverse the orientation of the arcs of $C$ in $G_c$. Then, the orientation of the arcs in $C$ is the same as the orientation of the arcs in $G_c^{'}$, and we delete these arcs from $D$. By iteratively taking one of the balanced cycles in $G(V,D)$ and reversing its orientation, $A$ can be transformed into $A'$. 


  \emph{Case 2: $r$ is odd}.
  Let $\Delta_{G_c}(v), \Delta_{G_c^{'}}(v)$ and $\Delta_{G(V,D)}(v)$ be the difference between the in and outdegree of vertex $v$ in the graphs $G_c, G_c^{'}$ and $G(V,D)$, respectively. If $\Delta_{G}(v) = \Delta_{G'}(v)$, then $\Delta_{G(V,D)}(v) = 0$ for $v \in V$. If this holds for every vertex, $G(V,D)$ can be decomposed into directed cycles and we can repeat the argument of case 1. So suppose that there is some vertex $v$ with $\Delta_{G}(v) \neq \Delta_{G'}(v)$. If $\Delta_{G}(v) > \Delta_{G'}(v)$, then $\Delta_{G(V,D)}(v) = 1$, and if $\Delta_{G}(v) < \Delta_{G'}(v)$, then $\Delta_{G(V,D)}(v) = -1$ (since $h$ and $h'$ are balanced and $r$ is odd). Similar to the proof in Lemma \ref{lemma:PR1}, we know that in $G(V,D)$ there exists a path between a vertex with $\Delta_{G(V,D)}(v) \geq 1$ and a vertex with $\Delta_{G(V,D)}(v) \leq -1$. If we reverse this path, all edges of the path have the same orientation as in $h^{'}$, and we can delete these edges from $D$. We can repeat this until either $D$ is empty or until $\Delta_{G(V,D)}(v) = 0$ for every vertex $v \in V$, in which case we can repeat the argument of case 1.
\end{proof}


For RR tournaments, connectivity can be obtained by considering cycle reversals of length 3 only \citep{knust2006balanced}. For iRR tournaments, cycles with at least $n$ edges can be required, e.g.\ if the match graph is a Hamiltonian cycle. 

\subsection{iPTS}

Perfect one-factorizations play a central role in the connectivity of RR neighborhoods since they serve as a counterexample to prove that PRS and PTS are color-wise disconnected \citep{januario2016new}. In contrast, when $r \leq n-3$, there exists a perfect one-factorization from which iPTS can escape to a non-perfect one-factorization. This is illustrated in \Cref{fig:EscapePerfectOneFactorization}.

\begin{figure}[!h]
  \centering
  \caption{Illustration of iPTS (iPRS) escaping from perfect one-factorization}
\label{fig:EscapePerfectOneFactorization}
  \begin{subfigure}{0.45\textwidth}
\centering
  \begin{tikzpicture}
  \foreach \i/\name\label in {0/3, 1/2, 2/1, 3/8, 4/7, 5/6, 6/5, 7/4} {
        \node[default] (\name) at (\i*360/8:2cm) {\name};
    }

    \draw[<-,edge4, very thick] (1) -- (8);
    \draw[<-,edge4, very thick] (2) -- (7);
    \draw[->,edge4, very thick] (3) -- (6);
    \draw[<-,edge4, very thick] (4) -- (5);

    \draw[->,edge5blurred] (2) -- (8);
    \draw[->,edge5blurred] (1) -- (3);
    \draw[->,edge5blurred] (4) -- (7);
    \draw[<-,edge5blurred] (5) -- (6);

    \draw[<-,edge1blurred] (4) -- (8);
    \draw[<-,edge1blurred] (3) -- (5);
    \draw[<-,edge1blurred] (2) -- (6);
    \draw[<-,edge1blurred] (7) -- (1);

    \draw[->,edge3blurred] (5) -- (8);
    \draw[->,edge3blurred] (4) -- (6);
    \draw[<-,edge3blurred] (3) -- (7);
    \draw[<-,edge3blurred] (1) -- (2);

    \draw[<-,edge2blurred] (7) -- (8);
    \draw[<-,edge2blurred] (1) -- (6);
    \draw[<-,edge2blurred] (2) -- (5);
    \draw[<-,edge2blurred] (3) -- (4);

  \draw[fictive_edge, dashed, very thick] (1) -- (4);
  \draw[fictive_edge, dashed, very thick] (1) -- (5);
  \draw[fictive_edge, dashed, very thick] (2) -- (3);
  \draw[fictive_edge, dashed, very thick] (2) -- (4);
  \draw[fictive_edge, dashed, very thick] (3) -- (8);
  \draw[fictive_edge, dashed, very thick] (5) -- (7);
  \draw[fictive_edge, dashed, very thick] (6) -- (8);
  \draw[fictive_edge, dashed, very thick] (6) -- (7);

  \end{tikzpicture}
  \caption{Initial perfect one-factorization,\\ obtained with the circle method 
  }
  \label{fig:EscapePerfectOneFactorization1}
\end{subfigure}
\begin{subfigure}{0.45\textwidth}
\centering
  \begin{tikzpicture}
  \foreach \i/\name\label in {0/3, 1/2, 2/1, 3/8, 4/7, 5/6, 6/5, 7/4} {
        \node[default] (\name) at (\i*360/8:2cm) {\name};
    }

    \draw[<-,edge4, very thick] (1) -- (8);
    \draw[fictive_edge, dashed, very thick] (2) -- (7);
    \draw[fictive_edge, dashed, very thick] (3) -- (6);
    \draw[<-,edge4, very thick] (4) -- (5);

    \draw[->,edge3, very thick] (2) -- (8);
    \draw[->,edge3, very thick] (1) -- (3);
    \draw[->,edge3, very thick] (4) -- (7);
    \draw[<-,edge3, very thick] (5) -- (6);

    \draw[<-,edge1blurred] (4) -- (8);
    \draw[<-,edge1blurred] (3) -- (5);
    \draw[<-,edge1blurred] (2) -- (6);
    \draw[<-,edge1blurred] (7) -- (1);

    \draw[->,edge5blurred] (5) -- (8);
    \draw[->,edge5blurred] (4) -- (6);
    \draw[<-,edge5blurred] (3) -- (7);
    \draw[<-,edge5blurred] (1) -- (2);

    \draw[<-,edge2blurred] (7) -- (8);
    \draw[<-,edge2blurred] (1) -- (6);
    \draw[<-,edge2blurred] (2) -- (5);
    \draw[<-,edge2blurred] (3) -- (4);
    
  \draw[fictive_edge, dashed, very thick] (1) -- (4);
  \draw[fictive_edge, dashed, very thick] (1) -- (5);
  \draw[<-,edge4, very thick] (2) -- (3);
  \draw[fictive_edge, dashed, very thick] (2) -- (4);
  \draw[fictive_edge, dashed, very thick] (3) -- (8);
  \draw[fictive_edge, dashed, very thick] (5) -- (7);
  \draw[fictive_edge, dashed, very thick] (6) -- (8);
  \draw[<-,edge4, very thick] (6) -- (7);

  \end{tikzpicture}
  \caption{Resulting non-perfect one-factorization (the union of green and purple edges form two cycles)}
  \label{fig:EscapePerfectOneFactorization2}
\end{subfigure}
\end{figure}
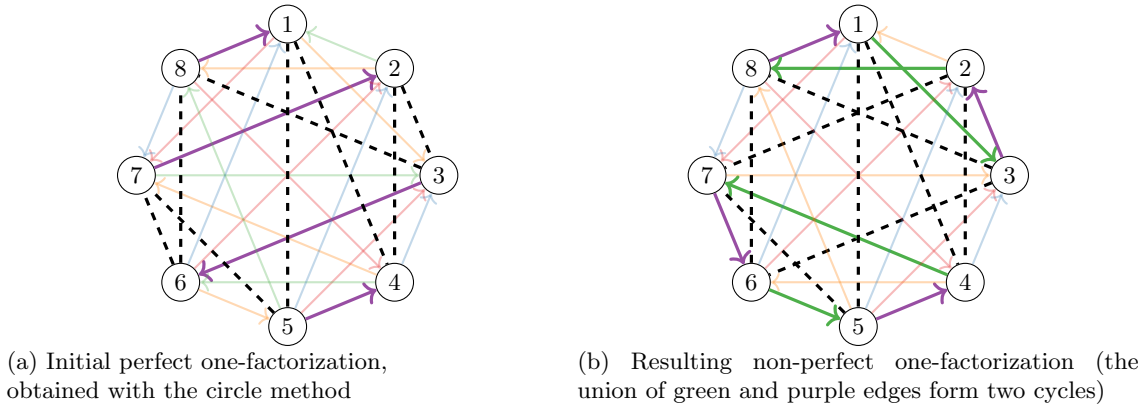

In \Cref{fig:EscapePerfectOneFactorization1}, an iPTS move on the incomplete lantarn with the edges $\{2,7\},\{6,7\},\{2,3\},\{3,6\}$ results in a one-factorization is not perfect anymore (see \Cref{fig:EscapePerfectOneFactorization2}). Hence, iPTS increases the solution space connectivity over the existing RR neighborhood structures. However, if $r=n-2$, we have the following result:

\begin{theorem}
  iPTS is color-wise disconnected when $r=n-2$.
\end{theorem}

\begin{proof}
  In case $r=n-2$, the uncolored edges form a one-factor $F'$. Hence, in case the union of the set of colored one factors and $F'$ is perfect, iPTS is equivalent to TS and thus only produces other perfect one-factorizations.
\end{proof}

\subsection{iPRS}\label{sec:iPRS-Connectivity}

For iPRS, we have seen two possible ways to guarantee that the match graph remains balanced: either we do iPRS-B or iPRS-U. Naturally, iPRS-B cannot be fully connected since it does not change the orientation of the matches. A natural question is whether iPRS-B, when complemented with CR, is connected.

\begin{theorem}\label{thm:HC}
  iPRS-B+CR is color-wise disconnected when $r=2$.
\end{theorem}

\begin{proof}
  Consider an iRR tournament where the match graph is the Hamiltonian cycle $(v_1,v_2,\dots,v_n,v_1)$ such that $c(\{v_i,v_{i+1}\}) = 1$ for odd $i$ and $c(\{v_i,v_{i+1}\}) = 2$ for even $i$. Moreover, let all $v_i$ with $i$ odd play home in round 1 and away in round 2 and all $v_i$ with $i$ even play away in round 1 and home in round 2. After any number of cycle reversals, we still have that all odd teams have the same home-away status in every round and that all even teams have the same home-away status in every round. As a result, teams of the same parity never have opposite home-away status, and hence any iPRS-B only matches teams with different parity. 
  \end{proof}



In the following, we will show that iPRS-U is fully connected in case that $r \leq \frac{n}{2}$. In order to do this, we first make a connection with perfect matchings. Recall that $G_s^u$ denotes the subgraph of a proper partially colored graph $G$ consisting of all uncolored edges and all edges colored $s$.

\begin{lemma}\label{lemma:iPRS-BM}
  Given a perfect matching $M'$ in $G^u_s$, there exists a sequence of iPRS-U moves so that round $s$ becomes identical to $M'$.
\end{lemma}

\begin{proof}
  First, let $M = \{e \in E: c(e) = s\}$. Given $M'$, which differs in at least one edge from $M$, consider the subgraph $G^u_s$ induced by the edges $E'$ that lie in the symmetric difference between $M$ and $M'$, i.e.\ $E' = ((M \setminus M') \cup (M' \setminus M))$. Every vertex in this graph must necessarily have degree 2. Therefore, $G^u_s$ is a collection of $s$-alternating cycles. Indeed, in any cycle, the edges must alternate between an edge colored $s$ in $M$ and an uncolored edge in  $M'$. 
  Hence, $M'$ can be obtained from $M$ by swapping the colored and the uncolored edges in each cycle, corresponding to a sequence of iPRS-U moves in $G^u_s$.
\end{proof}

We now recall a result from vertex coloring, which as we will show allows us to say something about the existence of a perfect matching in a graph. 

\begin{lemma}{\citet{chen1994equitable}}\label{lemma:chen}
    Let $G$ be a connected graph with maximum degree $k \geq \frac{n}{2}$. If $G$ is different from $K_n$ and in case $\frac{n}{2}$ odd different from $K_{\frac{n}{2}, \frac{n}{2}}$, then $G$ is equitable $k$-vertex colorable.
\end{lemma} 

In contrast to edge coloring, a vertex coloring is an assignment of colors to vertices instead of edges. An equitable $k$-vertex coloring of a graph $G(V,E)$ is an assignment of colors to $V$ such that no pair of adjacent vertices are assigned the same color, and the number of vertices in any two color classes differ by at most one. \citet{schmand2022greedy} show that Lemma~\ref{lemma:chen} implies that a ($\frac{n}{2}-1$)-regular graph $G$ contains a perfect matching if and only if either $\frac{n}{2}$ is even or $\frac{n}{2}$ is odd and $G$ is different from $K_{\frac{n}{2},\frac{n}{2}}$. Here, we show that Lemma~\ref{lemma:chen} implies the following stronger result:

\begin{lemma}\label{lemma:matching}
    Let $G$ be a graph with minimum and maximum degree $\frac{n}{2}-1$ and $\frac{n}{2}$, respectively, such that there is at least one vertex with degree $\frac{n}{2}$. Then, $G$ contains a perfect matching
\end{lemma}

\begin{proof}
    Consider the complement $\overline{G}$ of $G$, which is also a graph with minimum degree $\frac{n}{2}-1$ and maximum degree $\frac{n}{2}$. Hence, by Lemma~\ref{lemma:chen}, $\overline{G}$ can be equitably vertex colored with $\frac{n}{2}$ colors. This means that, for every one of these colors, there is exactly one pair of non-adjacent vertices in $\overline{G}$ (but adjacent in $G$) colored with this color. Hence, the subgraph formed with all the edges between pairs of vertices with the same color is a perfect matching of $G$.
\end{proof}

We are now ready to prove the following:

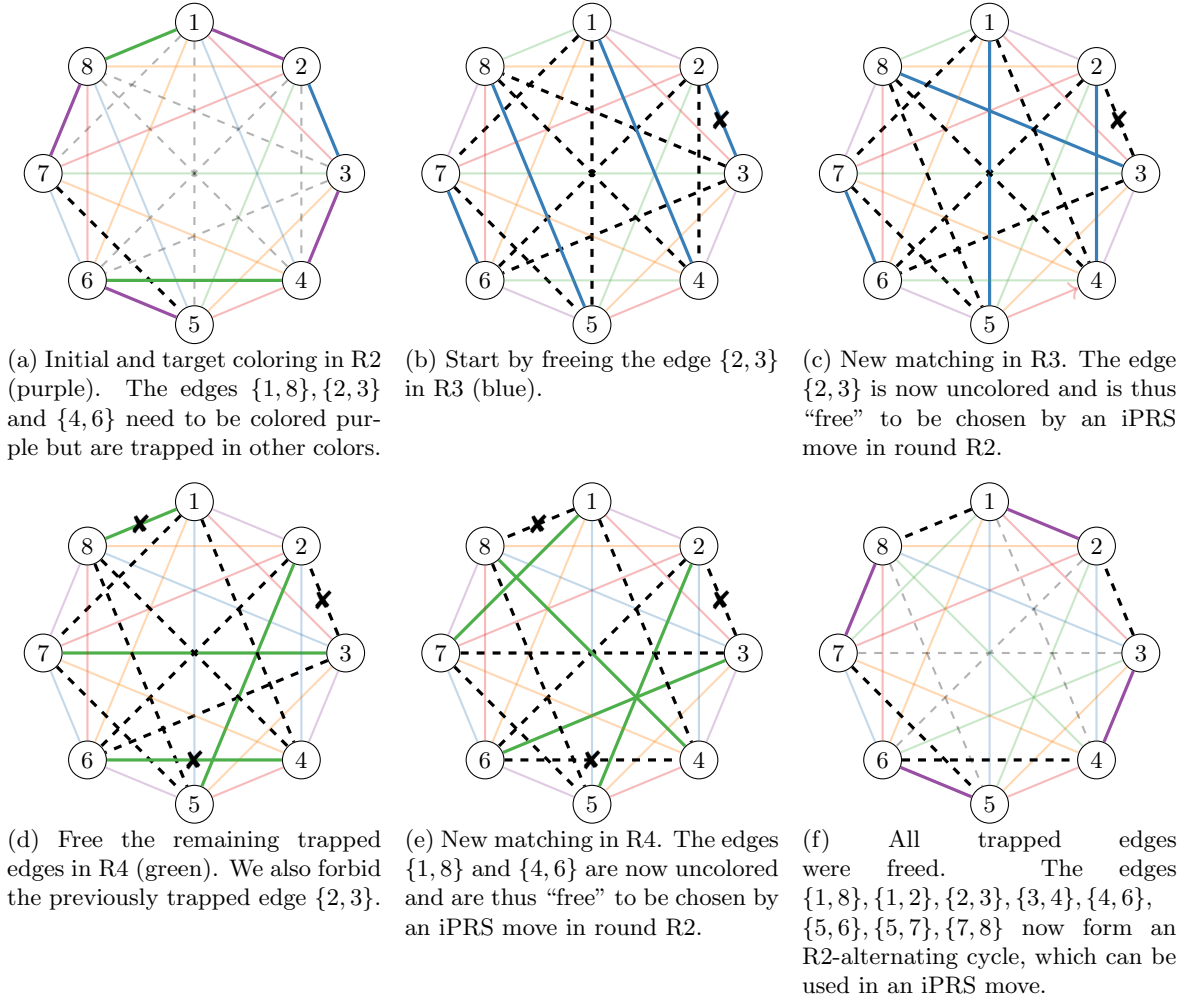
\begin{figure}[h!]
\centering
\caption{Example to illustrate the concept of freeing trapped edges, with round $l+1 = \text{R}2$, and $R^{\star}_{l+1} = \{\text{R}3,\text{R}4\}$. For illustrative purposes, all edges are unoriented. In practice, PR moves are used to restore balance. We further remark that in this example, $r=5 > 4 = \frac{n}{2}$.}
\label{fig:iPRS_con}
\begin{subfigure}{0.3\textwidth}
\centering
\begin{tikzpicture}
    \foreach \i/\name\label in {0/3, 1/2, 2/1, 3/8, 4/7, 5/6, 6/5, 7/4} {
        \node[default] (\name) at (\i*360/8:2cm) {\name};
    }
\draw[edge2, very thick] (3) -- (2);
\draw[edge1blurred] (3) -- (1);
\draw[fictive_edge, dashed, opacity=0.3] (3) -- (8);
\draw[edge3blurred] (3) -- (7);
\draw[fictive_edge, dashed, opacity=0.3] (3) -- (6);
\draw[edge5blurred] (3) -- (5);
\draw[edge4, very thick] (3) -- (4);

\draw[edge4, very thick] (2) -- (1); 
\draw[edge5blurred] (2) -- (8);
\draw[edge1blurred]  (2) -- (7);
\draw[fictive_edge, dashed, opacity=0.3] (2) -- (6);
\draw[edge3blurred] (2) -- (5);
\draw[fictive_edge, dashed, opacity=0.3] (2) -- (4);

\draw[edge3, very thick] (1) -- (8);
\draw[fictive_edge, dashed ,opacity=0.3] (1) -- (7);
\draw[edge5blurred] (1) -- (6);
\draw[fictive_edge, dashed ,opacity=0.3] (1) -- (5);
\draw[edge2blurred] (1) -- (4);

\draw[edge4, very thick] (8) -- (7);
\draw[edge1blurred] (8) -- (6);
\draw[edge2blurred] (8) -- (5);
\draw[fictive_edge, dashed,opacity=0.3] (8) -- (4);

\draw[edge2blurred] (7) -- (6);
\draw[fictive_edge, very thick, dashed] (7) -- (5);
\draw[edge5blurred] (7) -- (4);

\draw[edge4, very thick] (6) -- (5);
\draw[edge3, very thick] (6) -- (4);

\draw[edge1blurred] (5) -- (4);

\end{tikzpicture}
\caption{Initial and target coloring in R2 (purple). The edges $\{1,8\}, \{2,3\}$ and $\{4,6\}$ need to be colored purple but are trapped in other colors.}
\label{fig:iPRS_con1}
\end{subfigure}
\hspace{2pt}
\begin{subfigure}{0.3\textwidth}
\centering
\begin{tikzpicture}
    \foreach \i/\name\label in {0/3, 1/2, 2/1, 3/8, 4/7, 5/6, 6/5, 7/4} {
        \node[default] (\name) at (\i*360/8:2cm) {\name};
    }
\draw[edge2, very thick] (3) -- (2) node[midway, text=black]{\ding{56}};
\draw[edge1blurred] (3) -- (1);
\draw[fictive_edge, dashed, very thick] (3) -- (8);
\draw[edge3blurred] (3) -- (7);
\draw[fictive_edge, dashed, very thick] (3) -- (6);
\draw[edge5blurred] (3) -- (5);
\draw[edge4blurred] (3) -- (4);

\draw[edge4blurred] (2) -- (1); 
\draw[edge5blurred] (2) -- (8);
\draw[edge1blurred]  (2) -- (7);
\draw[fictive_edge, dashed, very thick] (2) -- (6);
\draw[edge3blurred] (2) -- (5);
\draw[fictive_edge, dashed, very thick] (2) -- (4);

\draw[edge3blurred] (1) -- (8);
\draw[fictive_edge, dashed, very thick] (1) -- (7);
\draw[edge5blurred] (1) -- (6);
\draw[fictive_edge, dashed, very thick] (1) -- (5);
\draw[edge2, very thick] (1) -- (4);

\draw[edge4blurred] (8) -- (7);
\draw[edge1blurred] (8) -- (6);
\draw[edge2, very thick] (8) -- (5);
\draw[fictive_edge, dashed, very thick] (8) -- (4);

\draw[edge2, very thick] (7) -- (6);
\draw[fictive_edge, very thick, dashed] (7) -- (5);
\draw[edge5blurred] (7) -- (4);

\draw[edge4blurred] (6) -- (5);
\draw[edge3blurred] (6) -- (4);

\draw[edge1blurred] (5) -- (4);

\end{tikzpicture}
\caption{Start by freeing the edge $\{2,3\}$ in R3 (blue).\phantom{a}\\\phantom{a}\\\phantom{a}}
\label{fig:iPRS_con2}
\end{subfigure}
\hspace{2pt}
\begin{subfigure}{0.3\textwidth}
\centering
\begin{tikzpicture}
    \foreach \i/\name\label in {0/3, 1/2, 2/1, 3/8, 4/7, 5/6, 6/5, 7/4} {
        \node[default] (\name) at (\i*360/8:2cm) {\name};
    }
\draw[fictive_edge, dashed, very thick] (3) -- (2) node[midway, text=black] {\ding{56}};
\draw[edge1blurred] (3) -- (1);
\draw[edge2, very thick] (3) -- (8);
\draw[edge3blurred] (3) -- (7);
\draw[fictive_edge, dashed, very thick] (3) -- (6);
\draw[edge5blurred] (3) -- (5);
\draw[edge4blurred] (3) -- (4);

\draw[edge4blurred] (2) -- (1); 
\draw[edge5blurred] (2) -- (8);
\draw[edge1blurred]  (2) -- (7);
\draw[fictive_edge, dashed, very thick] (2) -- (6);
\draw[edge3blurred] (2) -- (5);
\draw[edge2, very thick] (2) -- (4);

\draw[edge3blurred] (1) -- (8);
\draw[fictive_edge, dashed, very thick] (1) -- (7);
\draw[edge5blurred] (1) -- (6);
\draw[edge2, very thick] (1) -- (5);
\draw[fictive_edge, dashed, very thick] (1) -- (4);

\draw[edge4blurred] (8) -- (7);
\draw[edge1blurred] (8) -- (6);
\draw[fictive_edge, dashed, very thick] (8) -- (5);
\draw[fictive_edge, dashed, very thick] (8) -- (4);

\draw[edge2, very thick] (7) -- (6);
\draw[fictive_edge, very thick, dashed] (7) -- (5);
\draw[edge5blurred] (7) -- (4);

\draw[edge4blurred] (6) -- (5);
\draw[edge3blurred] (6) -- (4);

\draw[->,edge1blurred] (5) -- (4);

\end{tikzpicture}
\caption{New matching in R3. The edge $\{2,3\}$ is now uncolored and is thus “free" to be chosen by an iPRS move in round R2.}
\label{fig:iPRS_con3}
\end{subfigure}
\hspace{2pt}
\begin{subfigure}{0.3\textwidth}
\centering
\begin{tikzpicture}
    \foreach \i/\name\label in {0/3, 1/2, 2/1, 3/8, 4/7, 5/6, 6/5, 7/4} {
        \node[default] (\name) at (\i*360/8:2cm) {\name};
    }
\draw[fictive_edge, dashed, very thick] (3) -- (2) node[midway, text=black] {\ding{56}};
\draw[edge1blurred] (3) -- (1);
\draw[edge2blurred] (3) -- (8);
\draw[edge3, very thick] (3) -- (7);
\draw[fictive_edge, dashed, very thick] (3) -- (6);
\draw[edge5blurred] (3) -- (5);
\draw[edge4blurred] (3) -- (4);

\draw[edge4blurred] (2) -- (1); 
\draw[edge5blurred] (2) -- (8);
\draw[edge1blurred]  (2) -- (7);
\draw[fictive_edge, dashed, very thick] (2) -- (6);
\draw[edge3, very thick] (2) -- (5);
\draw[edge2blurred] (2) -- (4);

\draw[edge3, very thick] (1) -- (8) node[midway, text=black] {\ding{56}};
\draw[fictive_edge, dashed, very thick] (1) -- (7);
\draw[edge5blurred] (1) -- (6);
\draw[edge2blurred] (1) -- (5);
\draw[fictive_edge, dashed, very thick] (1) -- (4);

\draw[edge4blurred] (8) -- (7);
\draw[edge1blurred] (8) -- (6);
\draw[fictive_edge, dashed, very thick] (8) -- (5);
\draw[fictive_edge, dashed, very thick] (8) -- (4);

\draw[edge2blurred] (7) -- (6);
\draw[fictive_edge, very thick, dashed] (7) -- (5);
\draw[edge5blurred] (7) -- (4);

\draw[edge4blurred] (6) -- (5);
\draw[edge3, very thick] (6) -- (4) node[midway, text=black] {\ding{56}};

\draw[edge1blurred] (5) -- (4);

\end{tikzpicture}
\caption{Free the remaining trapped edges in R4 (green). We also forbid the previously trapped edge $\{2,3\}$.\\\phantom{a}\\\phantom{a}}
\label{fig:iPRS_con4}
\end{subfigure}
\hspace{2pt}
\begin{subfigure}{0.3\textwidth}
\centering
\begin{tikzpicture}
    \foreach \i/\name\label in {0/3, 1/2, 2/1, 3/8, 4/7, 5/6, 6/5, 7/4} {
        \node[default] (\name) at (\i*360/8:2cm) {\name};
    }
\draw[fictive_edge, dashed, very thick] (3) -- (2) node[midway, text=black] {\ding{56}};
\draw[edge1blurred] (3) -- (1);
\draw[edge2blurred] (3) -- (8);
\draw[fictive_edge, dashed, very thick] (3) -- (7);
\draw[edge3, very thick] (3) -- (6);
\draw[edge5blurred] (3) -- (5);
\draw[edge4blurred] (3) -- (4);

\draw[edge4blurred] (2) -- (1); 
\draw[edge5blurred] (2) -- (8);
\draw[edge1blurred]  (2) -- (7);
\draw[fictive_edge, dashed, very thick] (2) -- (6);
\draw[edge3, very thick] (2) -- (5);
\draw[edge2blurred] (2) -- (4);

\draw[fictive_edge, dashed, very thick] (1) -- (8) node[midway, text=black] {\ding{56}};
\draw[edge3, very thick] (1) -- (7);
\draw[edge5blurred] (1) -- (6);
\draw[edge2blurred] (1) -- (5);
\draw[fictive_edge, dashed, very thick] (1) -- (4);

\draw[edge4blurred] (8) -- (7);
\draw[edge1blurred] (8) -- (6);
\draw[fictive_edge, dashed, very thick] (8) -- (5);
\draw[edge3, very thick] (8) -- (4);

\draw[edge2blurred] (7) -- (6);
\draw[fictive_edge, very thick, dashed] (7) -- (5);
\draw[edge5blurred] (7) -- (4);

\draw[edge4blurred] (6) -- (5);
\draw[fictive_edge, dashed, very thick] (6) -- (4) node[midway, text=black] {\ding{56}};

\draw[edge1blurred] (5) -- (4);

\end{tikzpicture}
\caption{New matching in R4. The edges $\{1,8\}$ and $\{4,6\}$ are now uncolored and are thus “free" to be chosen by an iPRS move in round R2.\\\phantom{a}\\\phantom{a}}
\label{fig:iPRS_con5}
\end{subfigure}
\hspace{2pt}
\begin{subfigure}{0.3\textwidth}
\centering
\begin{tikzpicture}
    \foreach \i/\name\label in {0/3, 1/2, 2/1, 3/8, 4/7, 5/6, 6/5, 7/4} {
        \node[default] (\name) at (\i*360/8:2cm) {\name};
    }
\draw[fictive_edge, dashed, very thick] (3) -- (2);
\draw[edge1blurred] (3) -- (1);
\draw[edge2blurred] (3) -- (8);
\draw[fictive_edge, dashed, opacity=0.3] (3) -- (7);
\draw[edge3blurred] (3) -- (6);
\draw[edge5blurred] (3) -- (5);
\draw[edge4, very thick] (3) -- (4);

\draw[edge4, very thick] (2) -- (1); 
\draw[edge5blurred] (2) -- (8);
\draw[edge1blurred]  (2) -- (7);
\draw[fictive_edge, dashed, opacity=0.3] (2) -- (6);
\draw[edge3blurred] (2) -- (5);
\draw[edge2blurred] (2) -- (4);

\draw[fictive_edge, dashed, very thick] (1) -- (8);
\draw[edge3blurred] (1) -- (7);
\draw[edge5blurred] (1) -- (6);
\draw[edge2blurred] (1) -- (5);
\draw[fictive_edge, dashed, opacity=0.3] (1) -- (4);

\draw[edge4, very thick] (8) -- (7);
\draw[edge1blurred] (8) -- (6);
\draw[fictive_edge, dashed, opacity=0.3] (8) -- (5);
\draw[edge3blurred] (8) -- (4);

\draw[edge2blurred] (7) -- (6);
\draw[fictive_edge, very thick, dashed] (7) -- (5);
\draw[edge5blurred] (7) -- (4);

\draw[edge4, very thick] (6) -- (5);
\draw[fictive_edge, dashed, very thick] (6) -- (4);

\draw[edge1blurred] (5) -- (4);

\end{tikzpicture}
\caption{All trapped edges were freed. The edges $\{1,8\},\{1,2\},\{2,3\},\{3,4\},\{4,6\}$,\\$\{5,6\},\{5,7\},\{7,8\}$ now form an R2-alternating cycle, which can be used in an iPRS move.}
\label{fig:iPRS_con6}
\end{subfigure}
\end{figure}

\begin{theorem}\label{thm:ConnectivityMatching}
  iPRS-U is color-wise connected for iRR if $r \leq \frac{n}{2}$.
\end{theorem}

\begin{proof}
  Let $G(V,E,c)$ and $G'(V,E,d)$ be two graphs that are colored by the mappings $c$ and $d$, respectively. We define $E^{g}_k = \{e \in E: g(e) = k\}$, i.e.\ $E^{g}_k$ is the set of edges that are colored $k$ under coloring $g \in \{c,d\}$. Assume that the one-factors corresponding to colors in the set $R' = \{1,\dots,l\}$ with $l \leq r-1$ of $G$ and $G'$ are identical, i.e. $E_{s}^{c} = E_{s}^{d}$ for every $s \in R'$. We now show how to use iPRS-U to make the one-factors corresponding to round $l+1$ in $G$ and $G'$ identical. 
  
  First, suppose that all edges with color $l+1$ in $G'$ are either uncolored or have color $l+1$ in $G$. This means that $G_{l+1}^{u}$ contains the perfect matching identical to round $l+1$ in $G'$. By Lemma~\ref{lemma:iPRS-BM}, this matching can be found by a sequence of iPRS-U moves. For example, if we consider round R2 (purple edges) in \Cref{fig:BipartiteMatching_iRR1}, and we want to transform this matching to the matching depicted in \Cref{fig:BipartiteMatching_iRR3}, then a single iPRS-U move suffices to transform the matching in \Cref{fig:BipartiteMatching_iRR1} into the matching in \Cref{fig:BipartiteMatching_iRR3}. Then, we can increment $l$ by one, and consider the next color.
  
 However, consider now \Cref{fig:iPRS_con}. Here, we want to transform the matching in round R2 given by the edges $\{1,2\}, \{3,4\}, \{5,6\}, \{7,8\}$ into the matching consisting of the edges $\{1,8\}, \{2,3\}, \{4,6\}, \{5,7\}$, but the edge $\{2,3\}$ is “trapped" in R3 and the edges $\{1,8\}$ and $\{4,6\}$ are “trapped" in R4. We say these edges are trapped because these edges can only be used in an iPRS-U move in case they are uncolored or colored $l+1$ in $G$. For this purpose, define $R^{\star}_{l+1} = \{e \in E: c(e) \notin \{l+1,-1\} \ \land \ e \in E_{l+1}^d\}$. In other words, $R^{\star}_{l+1}$ is the set of trapped edges, i.e. the set of edges in $E$ that are colored with a color different from $l+1$ in $G$ but should be colored $l+1$ in $G'$. The proof will exploit Lemma~\ref{lemma:matching} to iteratively go over the rounds in $R^{\star}_{l+1}$ until all trapped edges are either uncolored or colored $l+1$ in $G$, such that the desired matching can be found by an iPRS move. In order to do this, we assume that we have already freed trapped edges from colors $k \in R^{\star}_{l+1}: k < q$. Define $G^{u\star}_{q}$ to be the graph equal to $G^u_{q}$ without edges trapped colored $m \in R_{l+1}^{\star}$ with $m \leq q$. Recall that we assumed that $r \leq \frac{n}{2}$. Then, the degrees of vertices in $G^{u\star}_{q}$, are as follows.

  First, if the edge incident to $v$ with color $q$ is a trapped edge, then this edge is excluded from $G^{u\star}_{q}$, so $v$ has a degree of at least $\frac{n}{2}-1$ in $G^{u\star}_{q}$. Secondly, if the edge incident to $v$ colored $q$ is not a trapped edge but $v$ was incident to a trapped edge in some color previously considered in $R^{\star}$, then $v$ has a degree of at least $\frac{n}{2}-1-1+1 = \frac{n}{2}-1$ in $G^{u\star}_{q}$. Indeed, since the edges colored with $l+1$ form a perfect matching, each node is incident to at most one trapped edge. Finally, if $v$ was not incident to a trapped edge from a color previously considered in $R^{\star}$, and it is not incident to a trapped edge with color $q$, then it has a degree of at least $\frac{n}{2}$ in $G^{u\star}_{q}$. Given a graph $G^{u\star}_{q}$, we now consider two cases:

  \emph{Case 1: $G^{u\star}_{q}$ contains at least one vertex with degree at least $\frac{n}{2}$.} 

  In this case, Lemma \ref{lemma:matching} asserts that a perfect matching in $G^{u\star}_{q}$ exists. Then, use iPRS-U to find this matching (followed by PR to restore the balance) and move to the next color, if any exists. If there exists no color anymore with trapped edges, then all the edges colored $l+1$ in $G'$ are either uncolored or colored $l+1$ in $G$. This procedure is illustrated in \Cref{fig:iPRS_con}.

  \emph{Case 2: $G^{u\star}_{q}$ contains no vertices with degree at least $\frac{n}{2}$.} 

  Since we assumed that $r \leq \frac{n}{2}$, in this case all vertices in $G^{u\star}_{q}$ have degree $\frac{n}{2}-1$. It is known that the only case in which there exists no perfect matching in $G^{u\star}_{q}$ is when $G^{u\star}_{q}$ is equal to $K_{\frac{n}{2}, \frac{n}{2}}$ and $\frac{n}{2}$ is odd \citep{schmand2022greedy}. In this case, observe that the uncolored edges form two odd sized (complete) cliques of size $\frac{n}{2}$. So suppose that this is the case (if not, Lemma~\ref{lemma:iPRS-BM} asserts that we can do a sequence of iPRS-U moves to find the desired matching). Then, we do a sequence of iPRS-U moves involving color $l+1$, with the goal of swapping edges colored $l+1$ and $q$ in $G$. Since $G^{u\star}_{q}$ is equal to $K_{\frac{n}{2}, \frac{n}{2}}$, there must necessarily exist an alternating cycle in $G^{u\star}_{l+1}$ with exactly four edges. Indeed, all the edges colored $l+1$ must connect the two odd cliques of size $\frac{n}{2}$. Let $V_1$ be the vertices in the first clique and $V_2$ be the vertices in the second clique. Then, there must exist two edges $\{u,v\}$ and $\{w,x\}$ colored $l+1$, such that $u,w \in V_1$ and $v,x \in V_2$. Since the cliques are complete, the edges $\{u,w\}$ and $\{v,x\}$ must be uncolored.
  Swapping the colored for the uncolored edges in this cycle results in $G^{u\star}_{q}$, without trapped edges, not being equal anymore to $K_{\frac{n}{2}, \frac{n}{2}}$. Hence, now we can use Lemmas~\ref{lemma:iPRS-BM} and~\ref{lemma:matching} to free the trapped edges of color $q$, and all the edges colored $l+1$ in $G'$ are now either uncolored or colored $l+1$ in $G$. 

To conclude the proof, set $l=1$ and repeat the procedure above until $l=r-1$. In case that $l=r-1$, there can be no trapped edges anymore, so the move iPRS-U can only find the one-factor identical to the one-factor with color $r$ in $G'$.
\end{proof}

\begin{corollary}
  If $r \leq \frac{n}{2}$, iPRS-U+CR is fully connected if $r$ is even and iPRS-U+CR+PR is fully connected if $r$ is odd.
\end{corollary}

\begin{proof}
  Given two partially colored balanced graphs $G(V,E,c)$ and $G'(V,E,d)$, with orientations $h$ and $h'$, respectively. By \Cref{thm:ConnectivityMatching}, iPRS-U can be used to transform $G(V,E,c)$ with orientation $h$ into $G'(V,E,d)$ with orientation $h^{''}$, and by \Cref{thm:ConnectivityCR}, CR+PR can be used to transform $h^{''}$ into $h'$, without modifying the colors of the edges.
  
\end{proof}

Similar to iPTS, in case that $r=n-2$, we have the following result:

\begin{theorem}\label{thm:DisconnectivityMatching_n-2}
  iPRS-U is color-wise disconnected when $r=n-2$.
\end{theorem}

\begin{proof}
  In case $r=n-2$, the uncolored edges form a one-factor $F'$. In case the union of the set of colored one factors and $F'$ is perfect, the union of any pair of one-factors in the set $\{F_1,\dots,F_{n-2},F'\}$ forms a Hamiltonian cycle. Hence, for each one-factor $F_c, c \leq n-2$, iPRS-U either returns $F_c$ or $F'$, thus only leading to other perfect one-factorizations. 
\end{proof}

Whether or not iPRS-U is connected for values of $r \in \{\frac{n}{2}+1,\dots,n-3\}$ is an open question. Since the edges in \Cref{fig:EscapePerfectOneFactorization} form an alternating cycle, this example also shows that iPRS-U is able to escape from perfect one-factorizations in case $r \leq n-3$. 

\section{Experimental results}\label{sec:Experimental_results}

In this section, the goal is to empirically assess the effectiveness of the neighborhood structures. We do this for two recently proposed problems that involve finding an iRR timetable. In particular, we want to investigate the added value of the newly introduced neighborhoods (see Sections~\ref{sec:i-PTS} and \ref{sec:i-PRS}) beyond the neighborhoods known from round robin scheduling. In addition, we compare our results to the best known solutions from the literature.

The two problems and the corresponding sets of instances are described in \Cref{sec:instances}.
In \Cref{sec:framework}, we discuss the metaheuristic framework in which we use our neighborhoods, while in \Cref{sec:setup} we discuss our experimental setup. Finally, \Cref{sec:results} shows the results, followed by a discussion in \Cref{sec:discussion}.

\subsection{Problem descriptions and instances}\label{sec:instances}

We consider two problems: the incomplete Traveling Tournament Problem (iTTP), introduced by \citet{devriesere2026iTTP}, and a problem which we call Youth Sports Timetabling Problem (YSTP), stemming from \citet{li2025beyond} and \citet{devriesere2025redesigning}. Both involve scheduling a single incomplete round robin tournament as described in \Cref{sec:Notation}. The objective in both problems is to minimize the total travel distance. In particular, we are given a matrix $(d_{ij})_{i,j \in T}$ such that $d_{ij}$ is travel distance from the venue of team $i$ to that of team $j$. The difference between the two problems, however, is that in YSTP teams immediately return home after every game, while in iTTP teams travel directly to the venue of their next game. Hence, in iTTP, if team $i$ plays two consecutive away games against $j$ and $k$, then it incurs a cost of $d_{ij} + d_{jk} + d_{ki}$, while in YSTP the cost is $2d_{ij} + 2d_{jk}$. Thus, a sequence of consecutive away games can be seen as a road trip. In addition, each of the two problems contains a number of constraints inherent to the problem. In the following, we describe these constraints, together with the sets of instances.

\subsubsection{iTTP}

In iTTP, the task is to find a timetable such that the travel distance implied by the road trips is minimized. In addition to constraints \textbf{C1}-\textbf{C3} of \Cref{sec:Notation}, teams play at least one and at most three consecutive home games in any four consecutive rounds. \Cref{tab:instancesTTP} shows the considered instances and their characteristics. For more details, we refer to \citet{devriesere2026iTTP}.

\begin{table}[ht!]
  \centering
  \caption{Overview of the iTTP instances and their characteristics \citep{devriesere2026iTTP}}
  \label{tab:instancesTTP}
  \small
    \setlength{\tabcolsep}{5pt}
  \begin{threeparttable}
  \begin{tabular}{lcr p{9cm}}
    \toprule
    Setting & No. teams & No.rounds & Characteristics \\
    \midrule
    NL & 16 & 4,8,12 & National League of Major League Baseball \\
    BRA & 24 & 6,12,18 & Brazilian soccer championship \\
    NFL & 32 & 8,16,24 & National Football League \\
    GAL & 40 & 10,20,30 & Teams are located in a three-dimensional `galaxy' \\
    CIRC & 40 & 10,20,30 & Teams are located on a circle \\
    CON & 40 & 10,20,30 & All distances are 1\\
    LINE & 40 & 10,20,30 & Teams are located on a line\\
    INCR & 40 & 10,20,30 & Teams are located on a line with increasing distances\\
    \bottomrule
  \end{tabular}
  \end{threeparttable}
\end{table}

\subsubsection{YSTP}

In YSTP, each team belongs to exactly one club $c \in C$, and teams of the same club share the same venue. 
We let $\gamma_{c,r}$ denote the venue capacity of club $c$ in round $r$. A venue capacity violation is incurred for each home game scheduled at club $c$ in round $r$ that exceeds the capacity specified by $\gamma_{c,r}$.
For each instance, a parameter $v^+$ is given corresponding to the maximum allowed overall capacity violations. Moreover, for each team a set is given indicating the opponents that that team is eligible to play against (assuming that if team $i$ is eligible to play against team $j$, $j$ is also eligible to play against $i$). Note that in this case, the set of eligible matches does not necessarily correspond to the complete graph $K_n$. However, we may still take $K_n$ as the input graph, and put a high cost on the edges corresponding to ineligible matches. Teams are allowed to play at most two consecutive home and away games. Depending on the instance type, further restrictions are imposed on the home-away assignment of teams (see further).

For YSTP, we consider two sets of instances originating from Belgian youth sports competitions. The first set, refered to as $\mathcal{I}_h$, involves instances of Belgian field hockey competitions, which are scheduled as an iRR tournament since the season 2023--24 \citep{devriesere2025redesigning}. It involves data of two indoor competitions I-23-24 and I-24-25, and four outdoor competitions O-21-22, O-22-23, O-23-24 and O-24-25. Secondly, we consider the instances used in \citet{li2025beyond}. Four of these instances, named U13, U15, U17 and U21, are based on real-life Belgium football competitions, and were first considered in \citet{toffolo2019sport}. Together with two smaller artificial instances, Ti and S, we refer to this set of instances as $\mathcal{I}_f$. Each instance of $\mathcal{I}_f$ appears in threefold: once with the club capacities constant over the rounds and two times with club capacities varying over the rounds in different ways. These club capacity settings are denoted as CONST, NC-1 and NC-2; we refer to \citet{li2025beyond} for further details. 

\begin{table}[ht!]
  \centering
  \caption{Overview of the youth sports instances and their characteristics}
  \label{tab:instances}
  \begin{threeparttable}
  \begin{tabular}{cccccc}
    \toprule
    Instance set & Instance & No. teams & No. clubs & No. rounds \\
    \bottomrule
    $\mathcal{I}_h$ & I-23-24 & 34 & 19 & 6 \\
    $\mathcal{I}_h$ & I-24-25 & 40 & 20 & 6 \\
    $\mathcal{I}_h$ & O-21-22 & 1474 & 90 & 10 \\
    $\mathcal{I}_h$ & O-22-23 & 680 & 87 & 10 \\
    $\mathcal{I}_h$ & O-23-24 & 818 & 88 & 10 \\
    $\mathcal{I}_h$ & O-24-25 & 718 & 90 & 10 \\
    \hdashline
    $\mathcal{I}_f$ & Ti & 16 & 10 & 14 \\
    $\mathcal{I}_f$ & S & 50 & 37 & 18 \\
    $\mathcal{I}_f$ & U21 & 64 & 55 & 14 \\
    $\mathcal{I}_f$ & U17 & 144 & 118 & 14 \\
    $\mathcal{I}_f$ & U13 & 184 & 138 & 14 \\
    $\mathcal{I}_f$ & U15 & 216 & 136 & 14 \\
    \bottomrule 
  \end{tabular}
  \begin{tablenotes}
    Note: each instance in $\mathcal{I}_f$ appears six times, once for each combination of club capacities (CONST, NC-1, NC-2) and maximum number of breaks (3, $\infty$).
  \end{tablenotes}
  \end{threeparttable}
\end{table}

Finally, for instances in $\mathcal{I}_f$, a number of other restrictions on the home-away assignment are given. First, we recall that in the literature a pair of consecutive home or away games is known as a break. Then, for instances in $\mathcal{I}_f$ the following restrictions are given: no team starts or ends with a break \textbf{(D1)}, there are at least $\lfloor \frac{r}{4} \rfloor$ and at most $\lceil \frac{r}{4} \rceil$ home and away games in the first and last $\frac{r}{2}$ rounds \textbf{(D2)}, and every team has at most $b^+$ breaks  \textbf{(D3)}. We consider instances from \citet{li2025beyond} with $b^+ = 3$ and without \textbf{(D3)} (in which case we write $b^+ = \infty$). Table~\ref{tab:instances} gives an overview of the YSTP instances.


\subsection{Metaheuristic framework}\label{sec:framework}

In order to move through the search space and asses the effectiveness of the neighborhoods presented in \Cref{sec:Neighborhoods}, we develop a modified Late Acceptance Hill Climbing (LAHC) algorithm \citep{burke2017late}. In this framework, each move in the neighborhood generates a candidate solution.  When combining several neighborhood structures, they are given equal probability to be selected.

LAHC is an extension of pure Hill Climbing where instead of comparing a candidate solution with the current best solution, it is compared with the best solution $l_h$ iterations before. The original LAHC proposed in \citet{burke2017late} requires setting the value of the parameter $l_h$, which is the only parameter and is called the \emph{history length}. Instances in their paper are solved with values of $l_h$ ranging between 1 and 50,000. Let $z(S)$ be the cost of a timetable $S$. LAHC takes as input an initial solution $S^{init}$, initializes a list $L$ of length $l_h$, with a value $z(S^{init})$ in every position, i.e.\ if we let $L(k)$ be the $k$'th position in the list $L$ for $k =1,\dots,l_h$. Let $S^{incumbent}$ and $S^{best}$ be the current and best found timetable, respectively, both initialized with $S^{init}$. In iteration $i$, a candidate solution $S$ is accepted if and only if 

\begin{equation}
  z(S) < L(i \ (\text{mod}) \ l_h) \ \text{or} \ z(S) \leq z(S^{incumbent}) \label{eq:LACH_acceptance}
\end{equation}

and the algorithm terminates after reaching a certain number of consecutive iterations without accepting a candidate solution. For a detailed overview of this algorithm, we refer to \citet{burke2017late}.

\RestyleAlgo{ruled}
\begin{algorithm}[hbt!]
  \caption{Pseudocode of the proposed ALAHC algorithm}\label{alg:lahc}
\KwData{Initial solution $S^{init}$}
\KwResult{Final solution $S^{best}$}
$S^{best} \gets S^{init}, S^{incumbent} \gets S^{init}$\;
$l_h \gets 1$\;
$L \gets (s^{init})$\;
$\rho \gets 1.005, \rho^{incr} \gets 0.005$\;
$i \gets 0, i^{idle} \gets 0$\;
\While{time limit not hit}
{
    $i \gets i+1$\;
    Construct candidate solution $S$\;

    \eIf{$z(S) \geq z(S^{incumbent})$}
    {
        $i^{idle} \gets i^{idle}+1$\;
    }{
        $i^{idle} \gets 0$\;
    }
    \If{$z(S) < L(i \ (\text{mod}) \ l_h) \   \text{or} \ z(S) < z(S^{incumbent})$}
    {
        $S^{incumbent} = S$\;
    }
    \If{$z(S) < L(i \ (\text{mod}) \ l_h)$}
    {
        $L(i \ (\text{mod}) \ l_h) = z(S)$\;
    }

    \If{$i^{idle} > 100,000$ and $l_h \times 1.5 < 100,000$}
    {
        $l_h \gets l_h \times 1.5$\;
        $i \gets 0, i^{idle} \gets 0$\;
        $S^{incumbent} \gets S^{best}$\;
        $\rho \gets \rho+\rho^{incr}$\;
        $L \gets (v_1,v_2,\dots,v_{l_h})$ where $v_i \in [z(S^{best}),z(S^{best})\times \rho]$;\
    }

    \If{$z(S^{incumbent}) < z(S^{best})$ or $l_h \times 1.5 \geq 100,000$}
    {
        $l_h \gets 10$\;
        $i \gets 0, i^{idle} \gets 0$\;
        $S^{incumbent} \gets S^{best}$\;
        $\rho \gets 1.005$\;
        $L \gets (v_1,v_2,\dots,v_{l_h})$ where $v_i \in [z(S^{best}),z(S^{best})\times \rho]$;\
    }
}
\end{algorithm}

The speed of convergence of the LAHC algorithm proposed in \citet{burke2017late} depends entirely on $l_h$. Here, we step away from this dependency and update $l_h$ dynamically as suggested by \citet{tercero2026multi}, which results in an algorithm which we call Adaptive Late Acceptance Hill Climbing (ALAHC). First, we set $l_h = 1$, which corresponds to pure Hill Climbing. We keep the acceptance criterion as in Equation \eqref{eq:LACH_acceptance} and, after hitting 100,000 consecutive non-improving iterations, we increase the value of $l_h$ with 50\% (rounded up) and reinitialize $L$ with values sampled from the interval $[S^{best}, S^{best}\times \rho]$, with $\rho$ a perturbation value. The variable $\rho$ is initialized with 1.005 and increased after every convergence with 0.005. In case a new best solution is found, or in case the history length happens to grow to a length of at least 100,000, the history length is decreased to 10 and $\rho$ is set back to its initial value of 1.005. This dynamic adjustment serves a dual purpose: it actively guides the algorithm to improve upon a strong solution while preventing the algorithm from drifting afar from promising regions. The algorithm stops after hitting a predetermined time limit, which in this paper is set to two hours. The advantage of this approach is that the algorithm becomes time-based; it fully utilizes the allotted time by preventing premature convergence. The pseudocode of this algorithm can be found in Algorithm~\ref{alg:lahc}. Parameter values were tuned manually. We only accept solutions that are feasible, i.e.\ solutions that satisfy all hard constraints as described in \Cref{sec:instances}.

ALAHC assumes that a starting solution is given. For instances of iTTP, we build one using Vizing's edge coloring algorithm as described in \citet{devriesere2026iTTP}. Since this approach does not guarantee a feasible solution for instances of YSTP, we find initial solutions by using what we call the Greedy Matching algorithm (GM) described in \citet{li2025beyond}, with exactly the same parameter settings. Both approaches are extremely fast.


With respect to implementing the neighborhoods, we always use DFS to find cycles but vary between DFS and BFS (Breadth First Search) when finding paths. BFS also runs in $\mathcal{O}(|V|+|E|)$ time, but in contrast to DFS, it always returns a path that is shortest in terms of number of edges. An advantage of BFS is therefore that it reverses fewer orientations compared to DFS, and may therefore be less disruptive. Conversely, DFS explores a wider range of timetables and may therefore be better suited to escape from local optima. Therefore, in case a path needs to be found, BFS is used with a 90\% probability and DFS is used with a 10\% probability. Preliminary results confirmed that using a higher weight for BFS is justified. 

\subsection{Experimental setup}\label{sec:setup}

We first discuss how the neighborhood structures are combined into various configurations of the metaheuristic framework.
Recall that our algorithm for finding incomplete lanterns (see \Cref{subsec:lanterns}) may also return colorful chordless. Therefore, in this work PTS $\subset$ iPTS. By definition, iPTS-CR $\subset$ iPTS, and iPRS-B $\subset$ iPRS-U. Therefore, in the configuration where we combine all neighborhood structures, we sample moves from the neighborhood structures TS, iPTS, PRS, iPRS-U and CR. PTS and iPTS-CR moves are included in iPTS, and arbitrary iPRS-B moves are included in iPRS-U. We also test five other configurations. First, we include only the neighborhoods TS, PRS and CR. We call this the Base configuration, because these are the neighborhoods that already exist in the literature; it allows us to measure the added value of the proposed neighborhoods. Next, we test the performance of using the single neighborhood structures iPTS, iPTS-CR and iPRS-U, together with the combination of moves CR and iPRS-B. We combine CR with iPRS-B since iPRS-B does not modify the home-away patterns of the teams. This way, the effectiveness of the individual neighborhoods can be assessed. For each configuration, we run the algorithm with the same arbitrary set of 10 different seeds. 

For iTTP instances, we benchmark the ALAHC algorithm with the currently best known solutions from the literature, which can be found in \citet{devriesere2026iTTP}.
For YSTP instances, we benchmark the ALAHC algorithm with solving a standard integer program, which can be found in \Cref{app:IP}, as well as with the algorithm GM proposed in \citet{li2025beyond}. 

All algorithms were programmed in C++ and compiled using GCC 13.2.0 with the optimization flag -O3. All experiments were run on a GNU/Linux based system with an AMD EPYC 7532 32-Core Processor running at 3.3GHz, provided with 8 threads and 64GB of RAM. In every experiment, a time limit of 2 hours was used. The integer program was solved with GUROBI 12.0. 



\subsection{Results}\label{sec:results}

In this section, we analyze the performance of the ALAHC algorithm with the proposed neighborhoods. The solutions are compared with the best known lower bound. For iTTP instances, the derivation of these lower bounds is discussed in \citet{devriesere2026iTTP}. For instances in $\mathcal{I}_f$, the lower bound is the best lower bound found by GUROBI when solving the integer program in \Cref{app:IP} with a time limit of 2 hours. For instances in $\mathcal{I}_h$, a tighter lower bound is obtained by summing the minimum travel distance in each competition separately, ignoring the capacity restrictions. 

\begin{table}[h!]
    \centering
    \setlength{\tabcolsep}{2pt} 
    \caption{Results iTTP instances}
    \label{tab:Results_TTP}
    \begin{threeparttable}
    \begin{tabularx}{\textwidth}{l @{\extracolsep{\fill}} r @{\extracolsep{\fill}} r @{\extracolsep{\fill}} rrrrrr}
        \toprule
         & & & \multicolumn{6}{c}{ALAHC} \\ 
         \cmidrule(lr){4-9}
        Instance & LB & Best-known & Base & iPTS & iPTS-CR & iPRS-U & CR+iPRS-B & All \\
        \midrule
NL16-4 & 25076 & 25076 & \textit{30801} & \textbf{25076} & \textbf{25076} & \textbf{25076} & 25254 & \textbf{25076} \\
NL16-8 & 57263 & 62125 & \textit{65868} & 59551 & 59728 & 60766 & 61486 & \textbf{59483} \\
NL16-12 & 92580 & 110037 & 106440 & 105025 & 100668 & \textit{107162} & 105253 & \textbf{100144} \\
BRA24-6 & 39127 & 41004 & \textit{55216} & \textbf{40718} & \textbf{40718} & 41074 & 40954 & \textbf{40718} \\
BRA24-12 & 98815 & 120615 & \textit{126366} & 106259 & \textbf{105577} & 109473 & 108588 & 106077 \\
BRA24-18 & 167657 & 212852 & \textit{204445} & 188155 & 183896 & 190340 & 187659 & \textbf{183310} \\
NFL32-8 & 70158 & 83771 & \textit{109172} & 75199 & 75662 & 79484 & 78631 & \textbf{74902} \\
NFL32-16 & 175051 & 215826 & \textit{241812} & \textbf{195824} & 197071 & 204706 & 200706 & 197375 \\
NFL32-24 & 297305 & 396264 & \textit{385818} & 343846 & 345519 & 349378 & 347685 & \textbf{341253} \\
CIRC40-10 & 560 & 814 & \textit{1494} & 714 & \textbf{688} & 794 & 768 & 696 \\
CIRC40-20 & 1760 & 2770 & \textit{3518} & 2330 & 2272 & 2424 & 2370 & \textbf{2238} \\
CIRC40-30 & 3620 & 5640 & \textit{5724} & 4686 & 4664 & 4786 & 4732 & \textbf{4622} \\
CON40-10 & 280 & 280 & \textit{287} & \textbf{280} & \textbf{280} & \textbf{280} & \textbf{280} & \textbf{280} \\
CON40-20 & 560 & 560 & \textit{569} & \textbf{560} & \textbf{560} & 561 & 561 & 561 \\
CON40-30 & 810 & 812 & \textit{840} & 833 & 833 & \textbf{832} & 835 & 834 \\
GAL40-10 & 23794 & 28201 & \textit{32323} & \textbf{25228} & 25518 & 26258 & 26296 & 25538 \\
GAL40-20 & 51992 & 63135 & \textit{68729} & \textbf{56757} & 57551 & 59163 & 58310 & 57507 \\
GAL40-30 & 81752 & 112270 & \textit{106530} & \textbf{93781} & 95007 & 95757 & 94324 & 94089 \\
INCR40-10 & 13492 & 17796 & \textit{36056} & 15056 & 15424 & 16874 & 16332 & \textbf{15050} \\
INCR40-20 & 47222 & 67078 & \textit{85696} & 55816 & \textbf{55616} & 60904 & 59078 & 56010 \\
INCR40-30 & 102336 & 142404 & \textit{140008} & \textbf{115622} & 117714 & 121626 & 118582 & 115912 \\
LINE40-10 & 678 & 884 & \textit{1870} & 788 & \textbf{760} & 868 & 828 & \textbf{760} \\
LINE40-20 & 2364 & 3370 & \textit{4366} & 2844 & 2802 & 3024 & 2936 & \textbf{2788} \\
LINE40-30 & 5120 & 6990 & \textit{7176} & 5852 & 5888 & 6110 & 5932 & \textbf{5800} \\
        \bottomrule
    \end{tabularx}
    \begin{tablenotes}
\item For every instance, the best found upper bound is marked in bold, while the worst found upper bound is italicized. 
\end{tablenotes}
\end{threeparttable}
\end{table}

\Cref{tab:Results_TTP} shows, for each iTTP instance, the best known lower bound (LB), the best known solution from the literature (Best-known), together with the best found solution over 10 random seeds of the ALAHC algorithm when a) using only TS, PRS and CR (Base), b) using only iPTS, c) using only iPTS-CR, d) using only iPRS-U, e) using both CR and iPRS-B and f) using TS, PRS, CR, iPTS and iPTS-U (All). As can be seen from the table, we are able to either match or find new best solutions for every iTTP instance. In particular, the configurations iPTS, iPTS-CR and All find better solutions in 29 out of the 30 instances, while iPRS-U and CR+iPRS-B find better solutions in 27 out of the 30 instances. The best known solution is found by iPTS in 9 cases, by iPTS-CR in 8 cases, by iPRS-U in 3 cases, only once by CR+iPRS-B, and finally in 14 cases by All. The distribution of the gaps of the solutions compared to the lower bound, for each instance and all 10 seeds, is shown in \Cref{fig:boxplots_ittp}. From this figure, we can see that the gaps are smallest in the configurations iPTS, iPTS-CR and All, and highest in the Base configuration. Moreover, on average the biggest improvement is found by All, where the best found solution is more than 22\% better than the best found solution by Base, while this is the lowest for iPRS-U, where on average the best found solution is “only" 19\% better.

\begin{figure}[h!]
     \centering
     \caption{Gaps iTTP instances}
     \label{fig:boxplots_ittp}
     \includegraphics[width=0.75\textwidth]{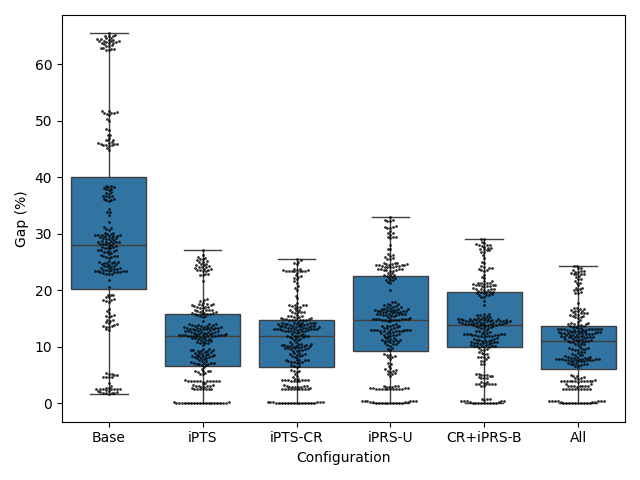}
 \end{figure}

\begin{table}[h!]
\centering
    \setlength{\tabcolsep}{2pt} 
    \caption{Results YSTP ($\mathcal{I}_f$) instances (football)}
    \label{tab:Results_football}
    \begin{threeparttable}
    \begin{tabularx}{\textwidth}{l @{\extracolsep{\fill}} r @{\extracolsep{\fill}} rr @{\extracolsep{\fill}} rrrrrr}
        \toprule
         & & & & \multicolumn{6}{c}{ALAHC} \\ 
         \cmidrule(lr){5-10}
        Instance & LB & IP & GM & Base & iPTS & iPTS-CR & iPRS-U & CR+iPRS-B & All \\
        \midrule
S-CONST-b3 & 4891 & \textit{4989} & 4964 & 4964 & 4911 & \textbf{4898} & 4942 & 4938 & 4922 \\
S-CONST & 4890 & \textbf{4891} & \textit{4938} & 4937 & 4905 & 4896 & 4934 & 4931 & 4922 \\
S-NC1-b3 & 4891 & 4937 & \textit{5003} & \textit{5003} & 4917 & \textbf{4902} & 4959 & 4949 & 4936 \\
S-NC1 & 4890 & \textbf{4891} & \textit{4976} & \textit{4976} & 4913 & 4900 & 4946 & 4946 & 4927 \\
S-NC2-b3 & 6340 & \textbf{6340} & \textit{6394} & 6392 & 6342 & 6341 & 6393 & 6386 & 6345 \\
S-NC2 & 6340 & \textbf{6340} & \textit{6356} & \textit{6356} & 6341 & \textbf{6340} & 6353 & 6353 & 6343 \\
U13-CONST-b3 & 4835 & \textit{19535} & 4978 & 4978 & 4894 & \textbf{4890} & 4976 & 4976 & 4973 \\
U13-CONST & 4835 & \textit{19053} & 4942 & 4942 & 4888 & \textbf{4880} & 4941 & 4942 & 4939 \\
U13-NC1-b3 & 4835 & \textit{8789} & 5017 & 5017 & 4919 & \textbf{4906} & 5016 & 5017 & 5016 \\
U13-NC1 & 4835 & \textit{14928} & 4996 & 4996 & 4907 & \textbf{4895} & 4996 & 4996 & 4994 \\
U13-NC2-b3 & 5757 & \textit{7404} & 6303 & 6303 & 5852 & \textbf{5836} & 6203 & 6301 & 5932 \\
U13-NC2 & 5757 & \textit{5894} & 5845 & 5845 & 5811 & \textbf{5803} & 5845 & 5845 & 5842 \\
U15-CONST-b3 & 3186 & -1 & \textit{3241} & \textit{3241} & 3206 & \textbf{3203} & 3239 & 3238 & 3238 \\
U15-CONST & 3186 & \textit{15221} & 3222 & 3222 & \textbf{3199} & 3203 & 3221 & 3221 & 3220 \\
U15-NC1-b3 & 3186 & -1 & \textit{3252} & \textit{3252} & \textbf{3204} & 3206 & 3250 & 3249 & 3250 \\
U15-NC1 & 3186 & \textit{12169} & 3230 & 3230 & \textbf{3204} & \textbf{3204} & 3229 & 3229 & 3229 \\
U15-NC2-b3 & 3876 & -1 & \textit{4367} & \textit{4367} & \textbf{3916} & \textbf{3916} & 4142 & 4366 & 3994 \\
U15-NC2 & 3876 & \textit{4916} & 3943 & 3943 & \textbf{3898} & \textbf{3898} & 3941 & 3943 & 3927 \\
U17-CONST-b3 & 4322 & \textit{13901} & 4431 & 4431 & 4371 & \textbf{4367} & 4428 & 4427 & 4427 \\
U17-CONST & 4322 & \textit{6685} & 4412 & 4412 & 4365 & \textbf{4360} & 4409 & 4408 & 4407 \\
U17-NC1-b3 & 4322 & \textit{7320} & 4494 & 4494 & 4401 & \textbf{4386} & 4493 & 4492 & 4465 \\
U17-NC1 & 4322 & \textit{8098} & 4480 & 4480 & 4392 & \textbf{4374} & 4477 & 4477 & 4454 \\
U17-NC2-b3 & 5677 & 5744 & \textit{5787} & \textit{5787} & 5714 & \textbf{5712} & 5786 & 5786 & 5751 \\
U17-NC2 & 5678 & \textbf{5678} & \textit{5734} & \textit{5734} & 5715 & 5707 & 5733 & 5733 & 5733 \\
U21-CONST-b3 & 3205 & 3259 & \textit{3351} & \textit{3351} & 3222 & \textbf{3216} & 3250 & 3248 & 3236 \\
U21-CONST & 3205 & \textbf{3205} & \textit{3262} & 3261 & 3220 & 3214 & 3251 & 3242 & 3234 \\
U21-NC1-b3 & 3205 & 3258 & \textit{3339} & \textit{3339} & 3233 & \textbf{3223} & 3275 & 3275 & 3252 \\
U21-NC1 & 3205 & \textbf{3205} & \textit{3323} & \textit{3323} & 3222 & 3219 & 3262 & 3263 & 3245 \\
U21-NC2-b3 & 4304 & \textbf{4304} & \textit{4457} & \textit{4457} & 4340 & 4334 & 4432 & 4455 & 4357 \\
U21-NC2 & 4304 & \textbf{4304} & \textit{4323} & \textit{4323} & 4309 & 4306 & 4322 & 4322 & 4315 \\
 \bottomrule
    \end{tabularx}
    \begin{tablenotes}
\item For every instance, the best found upper bound is marked in bold, while the worst found upper bound is italicized. For instances U15-CONST-b3, U15-NC1-b3 and U15-NC2-b3, no initial solution by IP was found, and hence the lower bounds for these instances are set equal to the lower bounds of the corresponding instances with no break limit.
\end{tablenotes}
\end{threeparttable}
\end{table}

A similar analysis is performed for instances in the set $\mathcal{I}_f$, whose results can be found in \Cref{tab:Results_football}. The integer program finds a proven optimal solution in 7 out of the 30 instances, all of which are instances from S and U21. This is expected, as these instances involve a relatively low number of teams compared to the number of rounds. However, for larger instances, the solutions found by IP are very poor. In 3 instances, IP even fails to find a feasible solution. In 27 out of the 30 instances in $\mathcal{I}_f$, the Base configuration is not able to find any improvement over the initial solution. This is expected, as the initial solutions found by GM are already very close to the lower bounds, and hence it is unlikely that better solutions may be found by only considering the same isomorphic match graph. Moreover, for YSTP instances, the objective function only changes when introducing new games, and hence only TS has an effect on the objective function value in the Base configuration. In contrast, the other configurations are often able to find better solutions. Notably, the configurations with only iPTS or iPTS-CR unequivocally find better solutions than the configuration that combines all neighborhoods. In particular, All never achieves the best known solution, while iPTS finds the best known solution for 5 instances, and iPTS-CR does so for 20 out of the 30 instances, including one which is proven optimal. This is also shown in \Cref{fig:boxplots_football}, where iPTS-CR is observed to achieve the lowest gaps. 
Moreover, iPTS-CR achieves the largest average improvement: its best-found solutions are 1.95\% better than those of Base. In contrast, the smallest improvement is observed for CR+iPRS-B, where the best-found solutions are on average only 0.5\% better.

\begin{figure}[h!]
     \centering
     \caption{Gaps football instances}
     \label{fig:boxplots_football}
     \includegraphics[width=0.75\textwidth]{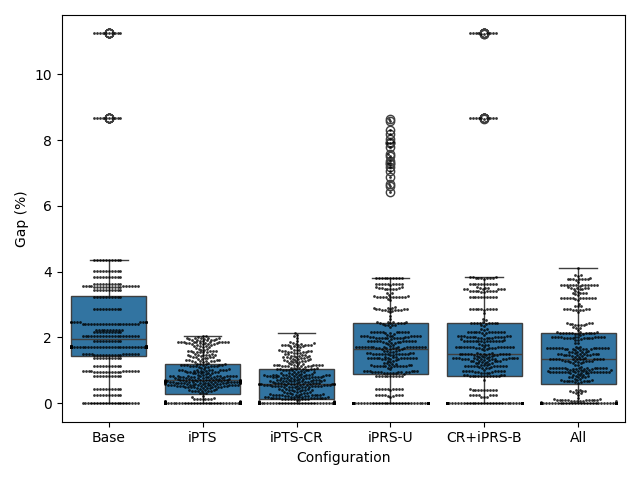}
 \end{figure}

\begin{table}[h!]
\centering
    \setlength{\tabcolsep}{2pt} 
    \caption{Results YSTP ($\mathcal{I}_h$) instances (field hockey)}
    \label{tab:Results_hockey}
    \begin{threeparttable}
    \begin{tabularx}{\textwidth}{l @{\extracolsep{\fill}} r @{\extracolsep{\fill}} rr @{\extracolsep{\fill}} rrrrrr}
        \toprule
         & & & & \multicolumn{6}{c}{ALAHC} \\ 
         \cmidrule(lr){5-10}
        Instance & LB & IP & GM & Base & iPTS & iPTS-CR & iPRS-U & CR+iPRS-B & All \\
        \midrule
I-23-24 & 2821 & \textbf{2821} & \textit{2903} & 2895 & \textbf{2821} & \textbf{2821} & 2825 & \textbf{2821} & \textbf{2821} \\
I-24-25 & 3980 & \textbf{3980} & \textit{4025} & 4021 & \textbf{3980} & \textbf{3980} & 3985 & 3983 & \textbf{3980} \\
O-21-22 & 150381 & \textit{289065} & 158058 & 158053 & 153803 & \textbf{153303} & 157437 & 155969 & 155283 \\
O-22-23 & 180310 & \textit{339237} & 190861 & 190857 & 182659 & \textbf{182091} & 188990 & 185311 & 183691 \\
O-23-24 & 220520 & \textit{391507} & 231584 & 231565 & 223464 & \textbf{222955} & 229963 & 226313 & 225091 \\
O-24-25 & 175592 & \textit{301517} & 187216 & 187052 & 178075 & \textbf{177489} & 184092 & 180558 & 179213 \\
 \bottomrule
    \end{tabularx}
    \begin{tablenotes}
\item For every instance, the best found upper bound is marked in bold, while the worst found upper bound is italicized. 
\end{tablenotes}
\end{threeparttable}
\end{table}

Finally, \Cref{tab:Results_hockey} reveals a similar conclusion for instances of the set $\mathcal{I}_h$. Here, IP, iPTS,iPTS-CR and All are able to find the optimal solutions for the (small) instances I-23-34 and I-24-25. For the other four instances, all configurations improve the solution found by GM, with the biggest improvements found by iPTS-CR, which finds the best known solution for all instances in $\mathcal{I}_h$. This is also shown in \Cref{fig:boxplots_hockey}. Moreover, on average the biggest improvement is found by iPTS-CR, where the best found solution is 3.33\% better than the best found solution by Base, while this is the lowest for iPRS-U, where on average the best found solution is only 1.16\% better.

\begin{figure}[h!]
     \centering
     \caption{Gaps hockey instances}
     \label{fig:boxplots_hockey}
     \includegraphics[width=0.75\textwidth]{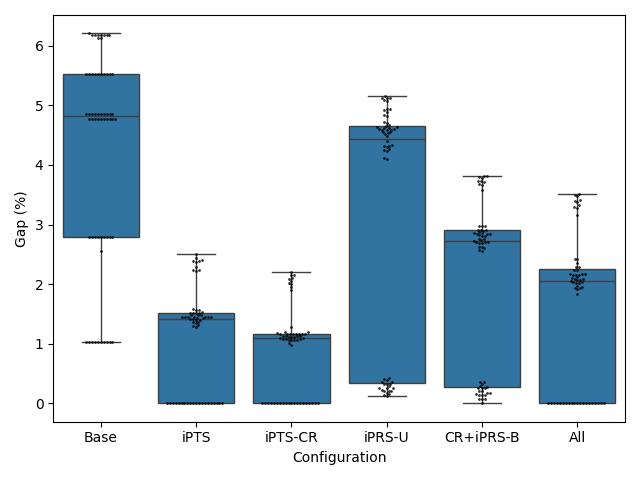}
 \end{figure}


\begin{table}[h]
\centering
\caption{Pairwise Nemenyi Test $p$-values for different instance sets}
\label{tab:Nemenyi}
\setlength{\tabcolsep}{2pt} 
\begin{threeparttable}
\scriptsize
\resizebox{\textwidth}{!}{
    \begin{tabular}{l  ccc  ccc  ccc  ccc  ccc}
    \toprule
    & \multicolumn{3}{c}{iPTS} & \multicolumn{3}{c}{iPTS-CR} & \multicolumn{3}{c}{iPRS-U} & \multicolumn{3}{c}{CR+iPRS-B} & \multicolumn{3}{c}{All} \\

    \midrule
    
    Configuration & \textit{iTTP} & $\mathcal{I}_f$ & $\mathcal{I}_h$ & \textit{iTTP} & $\mathcal{I}_f$ & $\mathcal{I}_h$ & \textit{iTTP} & $\mathcal{I}_f$ & $\mathcal{I}_h$ & \textit{iTTP} & $\mathcal{I}_f$ & $\mathcal{I}_h$ & \textit{iTTP} & $\mathcal{I}_f$ & $\mathcal{I}_h$ \\ 

    \cmidrule{1-1}
    \cmidrule(lr){2-4}
     \cmidrule(lr){5-7}
      \cmidrule(lr){8-10}
       \cmidrule(lr){11-13} 
        \cmidrule(lr){14-16}
    
    iPTS-CR   & $<.01$ & $.67$  & $>.05$ & \multicolumn{3}{c}{---} & \multicolumn{3}{c}{---} & \multicolumn{3}{c}{---} & \multicolumn{3}{c}{---} \\
    iPRS-U    & $<.01$ & $<.01$ & $<.01$ & $<.01$ & $<.01$ & $<.01$ & \multicolumn{3}{c}{---} & \multicolumn{3}{c}{---} & \multicolumn{3}{c}{---} \\
    CR+iPRS-B & $.09$  & $<.01$ & $>.05$ & $.06$  & $<.01$ & $<.01$ & $.92$  & $\ge.20$ & $.47$  & \multicolumn{3}{c}{---} & \multicolumn{3}{c}{---} \\
    All       & $.88$  & $<.01$ & $.84$  & $<.01$ & $<.01$ & $.29$  & $<.01$ & $\ge.20$ & $>.05$ & $<.01$ & $\ge.20$ & $.46$  & \multicolumn{3}{c}{---} \\
    Base      & $<.01$ & $<.01$ & $<.01$ & $<.01$ & $<.01$ & $<.01$ & $<.01$ & $<.01$ & $.25$  & $<.01$ & $<.01$ & $<.01$ & $<.01$ & $<.01$ & $<.01$ \\ 

    \bottomrule
    
    \end{tabular}
}
\end{threeparttable} 
\end{table}

We also test whether the solution quality significantly differs between the different configurations. As a first step, we perform a Friedman's test for each instance set. For all three instance sets, the null hypothesis that the configurations do not significantly differ from each other is rejected ($p < 0.01$). \Cref{tab:Nemenyi} shows that all the configurations with the proposed neighborhoods (except for iPRS-U for the instance set $\mathcal{I}_h$) yield significant improvements compared to Base. Moreover, this table confirms the superiority of iPTS-CR, as it either shows no significant difference with iPTS and All, or finds significantly better solution than all other configurations.

\subsection{Discussion}\label{sec:discussion}

The experiments in this section revealed that the configurations iPTS, iPTS-CR and All perform particularly well for TTP instances, while iPTS-CR dominates YSTP instances. The superiority of iPTS-CR for YSTP instances may have several reasons. First, in contrast to iTTP instances, YSTP instances impose many constraints on the home-away patterns of teams. In our framework, we only accept candidate solutions if they are feasible, suggesting that neighborhood structures with a minimum impact on the home-away patterns may fare well for YSTP instances. By complementing iPTS with CR within the lantern, we do exactly this, which may explain the improved performance of iPTS-CR over iPTS. 

Perhaps surprisingly, despite the fact that iPRS-U is color-wise connected in case that $r \leq \frac{n}{2}$, this configuration yielded, next to Base, the highest gaps. In contrast to iPTS-CR, iPRS-U may have a very disruptive effect on the home-away patterns of the teams, destroying effective road trips in the current solution of iTTP instances and violating hard constraints in the current solution of YSTP instances. Therefore, iPRS-U may be better fit as a perturbation move rather than an improvement neighborhood.
Moreover, when measuring the CPU time of the neighborhoods in the configuration All, iPRS-U consumes over 40$\%$ of the total CPU time in the configuration All. Hence, including this neighborhood is expensive, and may lead to significantly fewer iterations when using a time based stopping criterion. Instead of finding any alternating cycle, it may this be desirable to find minimum length alternating cycles, such that the acceptance rates of iPRS-U moves is increased.

At the same time, iPTS, IPTS-CR and All are seen to perform better than CR+iPRS-B, which suggests that simultaneously modifying the home-away pattern and opponents of teams in one single neighborhood might be more effective than having two separate neighborhoods for this.

\section{Conclusion}\label{sec:Conclusion}

In this work, we defined the iRR timetabling problem as a graph problem. We discussed the behavior of existing neighborhood structures and established that they are not connected for incomplete round robin tournaments.  We introduced two novel neighborhood structures, which we call incomplete Partial Team Swap and incomplete Partial Round Swap. It is shown that incomplete Partial Round Swap fully connects the solution space if the number of rounds is at most half the number of teams, which in practice is often the case. The effectiveness of the neighborhoods is validated on three sets of instances from the literature. Statistical tests confirm the added value of the proposed neighborhood structures. In particular, for each of the considered sets of instances, configurations with novel neighborhoods perform almost always significantly better than the configuration with existing neighborhoods only. Moreover, embedding the novel neighborhood structures in the ALAHC algorithm results in new best solution for all iTTP instances. Notably, for YSTP instances, using only incomplete Partial Team Swap produced the best results, suggesting that this neighborhood structure is able to effectively move to non-isomorphic timetables. Future work can focus on how to efficiently combine the various neighborhood structures, and investigate whether a more sophisticated framework, which for example allows neighborhoods to explore the infeasible solution space, can lead to further improvements. The solution space connectivity of incomplete Partial Team Swap, as well as incomplete Partial Round Swap in case $r > \frac{n}{2}$, also remain topics for future research.


\section*{Appendix}

\appendix

\section{Integer programming formulation YSTP}\label{app:IP}

Here, we present the integer programs used to compute timetables that minimize capacity violations and travel distance, for instances in the sets $\mathcal{I}_f$ and $\mathcal{I}_h$. Let $T_c$ be the set of teams belonging to club $c$. We use the variable $x_{ijr}$ which is 1 if team $i \in T$ plays against an eligible opponent $j \in O_i$ in round $r \in R$, and 0 otherwise. The variable $v_{cr}$ denotes the venue capacity violations of club $c \in C$ in round $r \in R$.

\begin{align}
  \text{min} \ \ & \sum_{i \in T}\sum_{j \in O_i}\sum_{r \in R}d_{ij}x_{ijr} & \label{ip:obj_travel}\\
  & \sum_{r \in R}(x_{ijr}+x_{jir}) \leq 1 & \forall i \in T, \ \forall j \in O_i \label{ip:c1}\\
  & \sum_{j \in O_i}(x_{ijr}+x_{jir}) = 1 & \forall i \in T, \ \forall r \in R \label{ip:c2}\\
  & \sum_{j \in O_i}\sum_{r \in R}x_{ijr} = \frac{|R|}{2} & \forall i \in T \label{ip:c3} \\
  & \sum_{j \in O_i}(x_{ijr-2}+x_{ijr-1}+x_{ijr}) & \forall i \in T, \ \forall r=3,\dots,|R| \label{ip:c4} \\
  & \sum_{j \in O_i}(x_{jir-2}+x_{jir-1}+x_{jir}) & \forall i \in T, \ \forall r=3,\dots,|R| \label{ip:c5} \\ 
  & \sum_{i \in T_c}\sum_{j \in O_i}x_{ijr} \leq \gamma_{cr} + v_{cr} & \forall c \in C, \ \forall r \in R \label{ip:c6} \\
  & \sum_{c \in C}\sum_{r \in R}v_{cr} \leq v^{+} \label{ip:c14}\\ 
  & x_{ijr} \in \{0,1\} & \forall i \in T, \ \forall j \in O_i, \ \forall r \in R \label{ip:x} \\
  & v_{cr} \in \mathbb{R} & \forall c \in C, \ \forall r \in R \label{ip:v}
\end{align}

The objective \eqref{ip:obj_travel} minimizes the total travel distance. Constraints \eqref{ip:c1} state that two teams face each other at most once, while constraints \eqref{ip:c2} state that, in each round, each team plays exactly one game against one of its eligible opponent. Constraints \eqref{ip:c3} guarantee that each team plays exactly $\frac{|R|}{2}$ home games (in all instances $|R|$ is even). Next, constraints \eqref{ip:c4} and \eqref{ip:c5} forbid a team from playing more than two consecutive home games and away games, respectively. Constraints \eqref{ip:c6} are the capacity constraints, with constraint \eqref{ip:c14} bounding the total capacity violation to $v^{+}$. Finally, expressions \eqref{ip:x}-\eqref{ip:v} give the variable domains.

In instances in $\mathcal{I}_f$, constraints \eqref{ip:c1}-\eqref{ip:v} are augmented with the following constraints:

\begin{align}
  & \sum_{j \in O_i}(x_{ijr} + x_{ijr+1}) \leq 1 & \forall i \in T, \ r \in \{1,|R|-1\} \label{ip:c7} \\
  & \sum_{j \in O_i}(x_{jir} + x_{jir+1}) \leq 1 & \forall i \in T, \ r \in \{1,|R|-1\} \label{ip:c8} \\
  & \sum_{j \in O_i}\sum_{r = s}^{s+\frac{|R|}{2}}x_{ijr} \geq \Bigl\lfloor \frac{|R|}{4} \Bigr\rfloor & \forall i \in T, \ \forall s \in \Bigl\{1,\frac{|R|}{2}\Bigr\} \label{ip:c9} \\
  & \sum_{j \in O_i}\sum_{r = s}^{s+\frac{|R|}{2}}x_{ijr} \leq \Bigl\lceil \frac{|R|}{4} \Bigr\rceil & \forall i \in T, \ \forall s \in \Bigl\{1,\frac{|R|}{2}\Bigr\} \label{ip:c10} \\
  & \sum_{j \in O_i}(x_{ijr-1}+x_{ijr}) \leq b_{ir} & \forall i \in T, \ \forall r=2,\dots,|R| \label{ip:c11} \\
  & \sum_{j \in O_i}(x_{jir-1}+x_{jir}) \leq b_{ir} & \forall i \in T, \ \forall r=2,\dots,|R| \label{ip:c12} \\
  & \sum_{r=2}^{|R|}b_{ir} \leq b^{+} & \forall i \in T \label{ip:c13} \\
  & b_{ir} \in \{0,1\} & \forall i \in T, \ \forall r \in R \label{ip:b}
\end{align}

Constraints \eqref{ip:c7}-\eqref{ip:c8} state that teams do not start with two consecutive home games and away games, respectively. Constraints \eqref{ip:c9}-\eqref{ip:c10} state that teams should play a balanced number of home and away games in each half. Next, a variable $b_{ir}$ is introduced that is 1 if team $i$ has a break in round $r$. Then, constraints \eqref{ip:c11}-\eqref{ip:c13} limit the number of breaks for each team. The domain of the variable $b_{ir}$ is given in expression \eqref{ip:b}.







\clearpage

\bibliography{bibliography}

\end{document}